\documentclass[a4paper,11pt]{article}

\usepackage{amsmath, amssymb, amsthm, amsfonts}

\usepackage{graphics}
\graphicspath{{Python/figures/2D/}}

\usepackage{psfrag, booktabs}
\usepackage{color, epsfig, float}
\usepackage[active]{srcltx}
\usepackage[bookmarksopen, colorlinks, linkcolor = blue, urlcolor = red,
            citecolor = red, menucolor = blue]{hyperref}
\usepackage{enumerate}            

\usepackage{float}
\floatstyle{plain}
\newfloat{algorithm}{thp}{alg}
\floatname{algorithm}{Algorithm}

\newtheorem{theorem}{Theorem}[section]
\newtheorem{lemma}[theorem]{Lemma}
\newtheorem{corol}[theorem]{Corollary}
\newtheorem{propo}[theorem]{Proposition}

\newtheorem{remark}[theorem]{Remark}
\newtheorem{assump}[theorem]{Assumption}

\newcommand\la{\lambda}

\newcommand{\ipl}{\langle}
\newcommand{\ipr}{\rangle}

\newcommand\summ{\textstyle\sum\limits}
\newcommand\T{\textstyle}
\newcommand\D{\displaystyle}

\newcommand\norm[1]{\|#1\|}
\newcommand\R{\mathbb{R}}

\newcommand\N{\mathbb{N}}

\DeclareMathOperator{\argmin}{arg\, min}

\begin{document}

\title{On inertial Levenberg-Marquardt type methods for solving nonlinear
ill-posed operator equations}
\setcounter{footnote}{1}

\author{
A.~Leit\~ao%
\thanks{\mbox{Department of Mathematics, Federal Univ.\,of St.\,Catarina,
        88040-900 Floripa, Brazil}}
\and
D.A.\,Lorenz%
\thanks{Center for Industrial Mathematics, University of Bremen, Bibliotheksstrasse 5, 28359 Bremen, Germany}
\and
J.C.~Rabelo%
\thanks{Department of Mathematics, Federal Univ.\,of Piaui,
        64049-550 Teresina, Brazil}
\and
M.\,Winkler%
\thanks{Center for Industrial Mathematics, University of Bremen, Bibliotheksstrasse 5, 28359 Bremen, Germany}
}
\date{\small \today}

\maketitle \vspace{-0.8cm}

\begin{abstract}
In these notes we propose and analyze an inertial type method for obtaining
stable approximate solutions to nonlinear ill-posed operator equations. The
method is based on the Levenberg-Marquardt (LM) iteration.
The main obtained results are: monotonicity and convergence for exact data,
stability and semi-convergence for noisy data. Regarding numerical experiments
we consider: i) a parameter identification problem in elliptic PDEs, ii) a
parameter identification problem in machine learning; the computational
efficiency of the proposed method is compared with canonical implementations
of the LM method.
\end{abstract}

{\let\thefootnote\relax\footnotetext{Emails:
\href{mailto:acgleitao@gmail.com}{\tt acgleitao@gmail.com}, \
\href{mailto:d.lorenz@uni-bremen.de}
     {\tt d.lorenz@uni-bremen.de}, \
\href{mailto:joelrabelo@ufpi.edu.br}{\tt joelrabelo@ufpi.edu.br},

\hskip0.3cm \href{mailto:maxwin@uni-bremen.de}{\tt maxwin@uni-bremen.de}.}
}

\bigskip

\bigskip

\section{Introduction} \label{sec:intro}


In a standard inverse problem scenario \cite{Bau87,EngHanNeu96,Kir96}, consider
Hilbert spaces $X$ and $Y$ and contemplate the challenge of deducing an unknown
quantity $x \in X$ from provided data $y \in Y$.
In other words, the task is to identify an unknown quantity of interest
$x$ (which cannot be directly accessed) relying on information derived
from a set of measured data $y$.

An essential aspect to note is that, in real-world applications, the
precise data $y \in Y$ is not accessible. Instead, only approximate
measured data $y^\delta \in Y$ is at our disposal, meeting the criteria of
\begin{equation}\label{eq:noisy-i}
    \norm{ y^\delta - y } \ \le \ \delta \, .
\end{equation}
Here, $\delta > 0$ represents the level of noise and we assume that $\delta$ (or an estimate thereof) is known. The available
noisy data $y^\delta \in Y$ are obtained by indirect measurements of
$x \in X$, this process being represented by the model
\begin{equation}\label{eq:inv-probl}
    F(x) \ = \ y^\delta ,
\end{equation}
where $F: X \to Y$, is a nonlinear, Fr\'echet differentiable, ill-posed
operator.

\subsection{State of the art}
\paragraph{The Levenberg-Marquardt (LM) method}

We recall a family of implicit iterative type methods for obtaining stable
approximate solutions to nonlinear ill-posed type operator equations as in
\eqref{eq:inv-probl}. The Levenberg-Marquardt (LM) type methods are
defined by
\begin{subequations} \label{def:LM}
\begin{equation}
x_{k+1}^\delta \ := {\rm argmin}_x \ \big\{
\norm{F(x_k^\delta) + F'(x_k^\delta)(x - x_k^\delta) - y^\delta}^2
+ \lambda_k \norm{x - x_k^\delta}^2 \big\} \, ,\ k = 0, 1, \dots \, ,
\end{equation}
what corresponds to defining $x_{k+1}^\delta$ as the solution of the
optimality condition
\begin{equation}
\big( A^*_k A_k + \lambda_k I \big) \, (x - x_k^\delta) \ = \
A^*_k \, \big( y^\delta - F(x_k^\delta) \big) \, ,\ k = 0, 1, \dots
\end{equation}
\end{subequations}
where $A_k := F'(x_k^\delta): X \to Y$ is the Fr\'echet derivative of
$F$ evaluated at $x_k^\delta$ and $A^*_k: Y \to X$ is the adjoint operator
to $A_k$. Additionally, $(\la_k)$ is a positive sequence of Lagrange
multipliers. The iteration starts at a given initial guess $x_0 \in X$.

In the case of linear ill-posed operator equations (i.e. $F(x) = Ax$;
notice that \eqref{eq:inv-probl} becomes $A x = y^\delta$) the LM method
reduces to the iterated Tikhonov method \cite{HG98, BLS20} (or proximal
point method \cite{Mar70, Roc76}), which correspond to defining
$x_{k+1}^\delta := \argmin_x \big\{ \norm{ Ax - y^\delta}^2 +
\la_k \norm{x - x_k^\delta}^2 \big\}$. The parameters $\la_k$ are
appropriately chosen Lagrange multipliers \cite{BLS20}.

The literature on LM type methods for inverse problems is extensive,
exploring various aspects, including regularization properties
\cite{Eng87, GS00, KN08, KalNeuSch08}, convergence rates
\cite{HG98, Sch93a}, {\em a posteriori} strategies for choosing the
Lagrange multipliers \cite{BLS20}, a cyclic version of the LM
method \cite{CBL11}, among others.

\paragraph{Inertial iterative methods}

Inertial iterative methods have been introduced by Polyak in~\cite{Pol64}
for the minimization of a smooth convex function $f$. The algorithm is
written as a two step method
\begin{align*}
w_k & = x_k + \alpha_k(x_k - x_{k-1})\\
x_{k+1} & = w_k - \lambda_k \nabla f(x_k) 
\end{align*}
where $\alpha_k$ is an extrapolation between $0$ and $1$ and $\lambda_k$
is a stepsize. The method is called the heavy-ball method as the extrapolation
can be motivated by a discretization of the dynamical system
$\ddot x(t) + \gamma \dot x(t) = -\nabla f(x(t))$ which models the dynamics
of a mass with friction driven by a potential $f$. The method has also
been extended to monotone operators, e.g. by Alvarez and Attouch
in~\cite{AA01} for the proximal point method and by Moudafi and Oliny
in~\cite{MO03} for the forward-backward method.

The heavy ball method achieves the optimal lower complexity bounds for
first order methods for smooth strongly convex functions~\cite{Nes04}.
For merely smooth function, a simple modification proposed by Nesterov
in~\cite{Nes83} achieves the lower complexity bounds also in this case.
The method reads as
\begin{align*}
w_k & = x_k + \alpha_k(x_k - x_{k-1})\\
x_{k+1} & = w_k - \lambda_k \nabla f(w_k) 
\end{align*}
and the only difference to the heavy ball method is that the gradient is
also evaluated at the extrapolated point. The performance relies on
a clever choice of the extrapolation sequence $\alpha_{k}$ such that it
approaches $1$ not too fast and not too slow. The method has been
extended to the forward backward case for convex optimization by Beck
and Teboulle~\cite{BeTe09} and further to monotone inclusions by Lorenz
and Pock~\cite{LP15}. Su, Boyd and Cand\'es~\cite{SBC16} related
Nesterov's method to the dynamical system
$\ddot x(t) + \tfrac{\alpha}{t}\dot x(t) = -\nabla f(x(t))$ which is
is similar to the heavy ball method but the damping $\alpha/t$ vanishes
asymptotically. The viewpoint of continuous dynamics was further
elaborated by Alvarez, Attouch, Bolte and Redont~\cite{AABR02} where
the authors proposed to analyze
\begin{align*}
\ddot x(t) + (\alpha I + \beta\nabla^2 f(x(t)))\ddot x(t) = -\nabla f(x(t))
\end{align*}
which they called \emph{dynamic inertial Newton system} (DIN). After time
disretization, this leads to an inertial Levenberg-Marquardt method similar
to the one we consider in this paper, but~\cite{AABR02} only analyzed the
continuous time system. Attouch, Peypoquet and Redont~\cite{APR16} combined
the DIN method with vanishing damping
\begin{align*}
\ddot x(t) + (\tfrac{\alpha}{t} I + \beta\nabla^2 f(x(t)))\ddot x(t) = -\nabla f(x(t)).
\end{align*}

\subsection{Contribution}

In these notes, we introduce and analyze an implicit inertial iteration,
here called {\it inertial Levenberg-Marquardt method} (inLM), which can
be construed as an extension of the LM method.
Our approach is connected to the inertial method put forth in 2001 by
Alvarez and Attouch \cite{AA01}.
In the case of linear ill-posed operator equations an approch analog to
the one addressed in this manuscript (namely the {\em inertial iterated
Tikhonov method}) is treated in \cite{RLM24}.

We suggest this implicit inertial method as a practical alternative for
computing robust approximate solutions to the ill-posed operator equation
\eqref{eq:inv-probl} and explore its numerical effectiveness.

The method under consideration consists in choosing appropriate
non-negative sequences $(\alpha_k)$, $(\la_k)$ and defining (at
each iterative step) the extrapolation
$w_k^\delta := x_k^\delta + \alpha_k (x_k^\delta - x_{k-1}^\delta)$;
the next iterate $x_{k+1}$ is than defined by
\begin{equation} \label{eq:inLM-step}
x_{k+1}^\delta \ := {\rm argmin}_x \ \big\{
\norm{F(w_k^\delta) + F'(w_k^\delta)(x - w_k^\delta) - y^\delta}^2
+ \lambda_k \norm{x - w_k^\delta}^2 \big\} \, ,\ k = 0, 1, \dots
\end{equation}
where $x_{-1} = x_0\in X$ are given. For obvious reasons we refer to this
implicit two-point method as {\em inertial Levenberg-Marquardt method} (inLM).
\medskip

\subsection{Outline}
The outline of the manuscript is as follows:
In Section~\ref{sec:iteration} we introduce and analyze the inLM method.
We prove a monotonicity result as well as convergence for exact data in
Section~\ref{ssec:2.2}, and discuss stability and semi-convergence
results in Section~\ref{ssec:2.3}.
In Section~\ref{sec:numerics} the inLM method is tested for two ill-posed
problems: i) a parameter identification problem in elliptic PDEs, ii) a
parameter identification problem in machine learning.
Section~\ref{sec:conclusions} is devoted to final remarks and conclusions.

\section{The inertial Levenberg-Marquardt method} \label{sec:iteration}

In this section we introduce and analyze the inLM method considered in
these notes. In Section~\ref{ssec:2.1} the inLM method is presented and
preliminary results are derived. A convergence result (in the exact data
case) is proven in Section~\ref{ssec:2.2}. Stability and semi-convergence
results (in the noisy data case) are proven in Section~\ref{ssec:2.3}.

This is the set of main assumptions that we impose on the operator $F$ and the data $y$:

\begin{enumerate}[({A}1)]
\item 
  The operator $F: X \to Y$ is continuously Fr\'echet differentiable.
  Moreover, there exist constants $C > 0$ and $\rho > 0$ such that
  $\norm{F'(x)} \leq C$, for all $x \in B_\rho(x_0)$.

\item 
  The operator $F$ satisfies the weak Tangential Cone Condition (wTCC) at
  $B_\rho(x_0)$ for some $\eta \in [0,1)$, i.e.
  \\[1ex]
  \centerline{$\norm{F(x') - F(x) - F'(x)(x' - x)} \ \leq \ \eta
    \, \norm{F(x')-F(x)} \, , \quad \forall \ x, \, x' \in B_\rho(x_0)$.}
\item 
  There exists $x^\star \in B_{\rho/2}(x_0)$ such that $F(x^\star) = y$,
  where $y \in Rg(F)$ is the exact data.
\item   There exists $q \in (\eta,1)$ such that $\la_k > q \, C^2
  \, (1 - q)^{-1}$, for $k = 0,1, \dots$
\end{enumerate}

\subsection{Description of method} \label{ssec:2.1}

To emphasize the fundamental principles that underlie the definition
of our method, we commence the discussion by examining the scenario
with exact data $y^\delta = y$, i.e. $\delta = 0$.
Denoting the current iterate by $x_k \in X$, for $k \geq 0$, the step
of the proposed inLM method consists in two parts: (i) compute $w_k \in X$,
according to
\begin{subequations} \label{def:inLM}
\begin{equation} \label{def:inLM-w}
w_k \ := \ x_k + \alpha_k \, (x_k - x_{k-1}) \, ;
\end{equation}
(ii) define the subsequent iterate $x_{k+1} \in X$ as the solution of
\begin{equation} \label{def:inLM-x}
\big( A^*_k A_k + \lambda_k I \big) \, (x - w_k) \ = \
A^*_k \, \big( y - F(w_k) \big) \, ,
\end{equation}
\end{subequations}
for $k = 0, 1, \dots$, were $A_k := F'(w_k): X \to Y$ is the Fr\'echet
derivative of $F$ at $w_k$ and $A^*_k: Y \to X$ is the adjoint operator
to $A_k$.
Here $x_0 \in X$ plays the role of an initial guess and $x_{-1} := x_0$.
Moreover, $(\alpha_k) \in [0,\alpha)$ for some $\alpha \in (0,1)$, and
$(\la_k) \in \R^+$ are given sequences.
Notice that, if $\alpha_k \equiv 0$ then $w_k = x_k$ in \eqref{def:inLM-w};
thus, \eqref{def:inLM-x} reduces to the standard LM iteration for exact
data, i.e. $x_{k+1}$ is defined as the solution of
$\big( A^*_k A_k + \lambda_k I \big) (x - x_k) \, = \, A^*_k \,
\big( y - F(x_k) \big)$, for $k = 0, 1, \dots$.

The careful reader observes that \eqref{def:inLM-x} is essentially the LM
iterative step \eqref{def:LM} starting from the extrapolation point $w_k$
instead of $x_k$. Notice that \eqref{def:inLM-x} is equivalent to computing
\begin{equation} \label{def:sk}
s_k \ := \ \big( A^*_k A_k + \lambda_k I \big)^{-1} \,
         A^*_k \, \big( y - F(w_k) \big)
\quad \text{ and setting } \quad
x_{k+1} := w_k + s_k
\end{equation}
($s_k$ is the iterative step of the inLM method).
It is straightforward to see that the first equation in \eqref{def:sk}
is equivalent to $G_k \, s_k = A^*_k (y - F(w_k))$, where
$G_k := (A^*_k A_k + \la_k I): X \to X$ is a positive definite operator
with spectrum contained in the interval $[\la_k , \la_k + \norm{A}^2]$.
Consequently, since $\la_k > 0$, the iterate $x_{k+1}$ is uniquely defined
by \eqref{def:inLM-x}.

We present the inLM method in algorithmic form in Algorithm~\ref{alg:init-exact}.

\begin{algorithm}[h!]
\begin{center}
\fbox{\parbox{13.5cm}{
$[0]$ choose an initial guess $x_0 \in X$; \ 
      $w_0 = x_0$; \ $k := 0$;
\medskip

$[1]$ choose $\alpha \in [0,1)$ \ and \ $(\la_k)_{k\geq0} \in \R^+$;
\medskip

$[2]$ {\bf for \ $k = 0, 1, \dots$ \ do}
\smallskip

\ \ \ \ \ \ \ {\bf if \ $\big( \norm{F(w_k) - y} > 0 \big)$ \ then}
\smallskip

\ \ \ \ \ \ \ \ \ \ $[2.1]$ $A_k := F'(w_k)$; \
          compute \,$s_k \in X$ as the solution of
\smallskip

\centerline{$\big( A^*_k A_k + \lambda_k I \big) \, s_k
\ = \ A^*_k \, \big( y - F(w_k) \big)$;}
\smallskip

\ \ \ \ \ \ \ \ \ \ $[2.2]$ $x_{k+1} := w_k + s_k$;
\smallskip

\ \ \ \ \ \ \ \ \ \ $[2.3]$ choose \,$\alpha_{k+1} \in [0,\alpha]$; \
          $w_{k+1} := x_{k+1} + \alpha_{k+1} (x_{k+1} - x_k)$;
\smallskip

\ \ \ \ \ \ \ {\bf else}
\smallskip

\ \ \ \ \ \ \ \ \ \ $[2.4]$ $s_k := 0$; \ $x_{k+1} := w_k$; \ {\bf break};
\smallskip

\ \ \ \ \ \ \ {\bf end if};
\smallskip

\ \ \ \ \,{\bf end for};
\smallskip
} }
\end{center} \vskip-0.5cm
\caption{Inertial Levenberg-Marquardt method in the exact data case.}
\label{alg:init-exact}
\end{algorithm}

\begin{remark}[Comments on Algorithm~\ref{alg:init-exact}] \label{rem:station}
This algorithm generates infinite sequences $(x_k)_{k\in\N}$ and
$(w_k)_{k\in\N}$ if and only if $F(w_k) \not= y$, for all $k \in \N$.
Indeed, if $F(w_{k_0}) = y$ for some $k_0 \in \N$ in Algorithm~%
\ref{alg:init-exact}, the iteration stops at Step~[2.4] after computing
\,$x_0, \dots, x_{k_0+1}$ \,and \,$w_0, \dots, w_{k_0}$.

The operators $A_k := F'(w_k) \in \mathcal{L}(X,Y)$ and $A_k^* \in
\mathcal{L}(Y,X)$ do not have to be explicitely known (see the inverse
problem in Section~\ref{ssec:num-nnp}).
The linear system in Step~[2.1] can be solved, e.g., using the Conjugate
Gradient (CG) mehthod; in this case it is enough to know only the action of
$A_k$ and $A_k^*$.
\end{remark}

In the remaining of this subsection we establish preliminary
properties of the sequences $(x_k)$, $(w_k)$ generated by Algorithm~%
\ref{alg:init-exact}. The first result, stated in Lemma~\ref{lemma:aux},
follows directly from the definition of $w_k$ in \eqref{def:inLM-w} (see
also Step~[2.4] of Algorithm~\ref{alg:init-exact}), while in Lemma~%
\ref{lemma:dk-and-ineq} some useful inequalities are derived.

\begin{lemma} \label{lemma:aux}
Let (A1) hold and $(x_k)$, $(w_k)$ be sequences generated by Algorithm~%
\ref{alg:init-exact}. Thus
\begin{equation} \label{eq:w}
\norm{w_k - x}^2 \, = \, (1+\alpha_k) \norm{x_k - x}^2 - \alpha_k
\norm{x_{k-1} - x}^2 + \alpha_k (1+\alpha_k) \norm{x_k - x_{k-1}}^2 ,
\ k \geq 1
\end{equation}
for $x \in X$.
\end{lemma}
\begin{proof}
See \cite[Lemma~2.2]{RLM24} for a complete proof.
\end{proof}

\begin{lemma} \label{lemma:dk-and-ineq}
Let (A1) hold and $(x_k)$, $(w_k)$ be sequences generated by Algorithm~%
\ref{alg:init-exact}.
Define $A_k := F'(w_k)$ and $D_k := F(w_k) + A_k (x_{k+1}-w_k) - y$. The
following assertions hold true:
\\[1ex]
a) $D_k \, = \, \la_k (A_k A^*_k + \la_k I)^{-1} (F(w_k) - y)$;
\\[1ex]
b) $w_k - x_{k+1} \ = \ \la_k^{-1} A^*_k \, [F(w_k) + A_k
(x_{k+1} - w_k) - y]$;
\\[1ex]
c) Additionally, if (A4) holds, we have
\,$q \norm{F(w_k) - y} \, \leq \, \norm{D_k} \, \leq \, \norm{F(w_k)-y}$;
\\[1ex]
d) Additionally, if (A2) holds and $w_k$, $x_{k+1} \in B_\rho(x_0)$, we have
\\[1ex]
\centerline{$(1-\eta) \norm{F(x_{k+1}) - y} \ \leq \ (1+\eta) \norm{F(w_k)-y}$.}
\end{lemma}
\begin{proof}
{\bf Assertions (a) and (b):} 
From Steps~[2.1] and~[2.2] of Algorithm~\ref{alg:init-exact} follow
\begin{equation} \label{eq:stepLM}
A^*_k \, [F(w_k) + A_k (x_{k+1}-w_k) - y] + \la_k (x_{k+1} - w_k) \ = \ 0
\end{equation}
(see also \eqref{eq:inLM-step}).
Consequently, $A_k A^*_k \, D_k  + \la_k A_k (x_{k+1} - w_k) = 0$, from
where we obtain
$$
A_k A^*_k \, D_k  + \la_k A_k (x_{k+1} - w_k) + \la_k (F(w_k) - y)
- \la_k (F(w_k) - y) \ = \ 0 \, .
$$
Thus, $(A_k A^*_k + \la_k I) \, D_k  = \la_k (F(w_k) - y)$ and
Assertion~(a) follows.

Assertion~(b) is an immediate consequence of \eqref{eq:stepLM}.
\smallskip

\noindent {\bf Assertion (c):}
If (A1) and (A4) hold, we conclude from Assertion~(a) together with the
fact $\sigma(A_kA^*_k + \la_k I) \subset [\la_k, \la_k + C^2]$ that
$$
\frac{\la_k}{\la_k + C^2} \norm{F(w_k) - y} \, \leq \, \norm{D_k}
\, = \,
\norm{\la_k (A_k A^*_k + \la_k I)^{-1} (F(w_k) - y)}
\, \leq \, \norm{F(w_k) - y} \, .
$$
From this inequality Assertion~(c) follows.
\smallskip

\noindent {\bf Assertion (d):}
From the definition of $D_k$ follows
$$
\norm{F(x_{k+1}) - y} \ \leq \
\norm{F(x_{k+1}) - F(w_k) - A_k (x_{k+1}-w_k)} + \norm{D_k} \, .
$$
Thus, it follows from (A2) and Assertion~(c)
\begin{eqnarray*}
\norm{F(x_{k+1}) - y}
& \leq & \eta \, \norm{F(x_{k+1}) - F(w_k)} + \norm{F(w_k) - y} \\
& \leq & \eta \, \norm{F(x_{k+1}) - y} + (1 + \eta) \norm{F(w_k) - y}) ,
\end{eqnarray*}
proving Assertion~(d).
\end{proof}

\begin{assump} \label{ass:alpha}
Given $\alpha\in[0,1)$ and a convergent series $\sum_k \theta_k$ of
nonnegative terms, let 
$$
\alpha_k :=
        \left\{ \begin{array}{cl}
          \min \left\{ \dfrac{\theta_k}{\norm{x_k-x_{k-1}}^2} ,
          \dfrac{\min\{ \theta_k ,\ \rho-\norm{x_k-x_0}\}}
          {\norm{x_k-x_{k-1}}} ,
          \alpha \right\}
                & , \ {\rm if } \ \norm{x_k - x_{k-1}} > 0 \\[2ex]
          0     & , \ {\rm otherwise}
        \end{array} \right.
\, , \ k \geq 1 \, . 
$$
\end{assump}

For simplicity of the presentation we assume, for the rest of this section,
that $\theta_k = 1/k^2$.

\begin{remark} \label{rem:alpha}
Should the sequence $(\alpha_k)$ of inertial parameters be chosen in
accordance with Assumption~\ref{ass:alpha}, two immediate consequences
ensue, namely:
\\[1ex]
a) \ If $x_{k-1}$, $x_k \in B_\rho(x_0)$, then $w_k \in B_\rho(x_0)$
as well. Indeed, from \eqref{def:inLM-w} follows
\\[1ex]
\centerline{$
\norm{w_k - x_0} \leq \norm{x_k - x_0} + \alpha_k \norm{x_k - x_{k-1}} <
\norm{x_k - x_0} + \dfrac{\rho-\norm{x_k-x_0}}{\norm{x_k-x_{k-1}}}
\norm{x_k - x_{k-1}} = \rho$}
\\[1ex]
(if $x_{k-1} = x_k$ holds, then \eqref{def:inLM-w} implies $w_k = x_k
\in B_\rho(x_0)$).
\medskip

\noindent
b) \ $\sum_{k \geq 0} \alpha_k \norm{x_k - x_{k-1}}^2$ is summable since by
$\alpha_k\leq \theta_{k}/\norm{x_{k}-x_{k-1}}^{2}$ it holds that
$\alpha_k\norm{x_k-x_{k-1}}^{2}\leq \theta_{k}$ and $\sum_{k\geq 0}\theta_{k}$ is summable by assumption.
\end{remark}

In the next proposition we compare the squared distances $\norm{w_k-x^*}^2$
and $\norm{x_{k+1} - x^*}^2$, where $x^*$ is any solution of $F(x) = y$ inside
the ball $B_\rho(x_0)$.

\begin{propo} \label{prop:gain}
Let (A1) -- (A4) hold and $(x_k)$, $(w_k)$ be sequences generated by
Algorithm~\ref{alg:init-exact} (with $(\la_k)$ and $(\alpha_k)$ chosen
as in Steps~[1] and [2.3] respectively). If $w_k \in B_\rho(x_0)$ then
$$
\norm{w_k - x^*}^2 - \norm{x_{k+1} - x^*}^2 \ \geq \
\norm{w_k - x_{k+1}}^2 + 2 (q - \eta) \la_k^{-1} \norm{D_k}
\norm{F(w_k)-y} \, , \ k \geq 0
$$
for any $x^* \in B_\rho(x_0)$ solution of $F(x) = y$,
\end{propo}
\begin{proof}
From Lemma~\ref{lemma:dk-and-ineq}~(b) follows
\begin{eqnarray*}
\|w_k - x\|^2 & \!\!\!\!\!\!\! - & \!\!\!\!\!\!\! \|x_{k+1} - x\|^2
\ = \ \|w_k - x_{k+1}\|^2 + 2\ipl w_k - x_{k+1} , x_{k+1} - x \ipr
\\
& = & \| w_k - x_{k+1} \|^2 + 2 \la_k^{-1}
  \ipl A^*_k [F(w_k) + A_k (x_{k+1} - w_k) - y] , x_{k+1} - x \ipr
\\
& = & \| w_k - x_{k+1} \|^2 + 2 \la_k^{-1}
  \ipl D_k , A_k (x_{k+1} - x) \ipr
\\
& = & \| w_k - x_{k+1} \|^2 + 2 \la_k^{-1}
  \ipl D_k , A_k \, (x_{k+1} - x \pm w_k) \pm F(w_k) \pm y \ipr
\\
& = & \| w_k - x_{k+1} \|^2 + 2 \la_k^{-1} \big[
  \norm{D_k}^2 +
  \ipl D_k , A_k \, (w_k - x) - F(w_k) + y ) \ipr \big] \, ,
\end{eqnarray*}
for $x \in X$ and $k \geq 0$. From this equation with $x = x^*$,
(A2) and Lemma~\ref{lemma:dk-and-ineq}~(c) follows
\begin{eqnarray*}
\| w_k - x \|^2  - \| x_{k+1} - x\|^2
& \!\!\!\! = & \!\!\!\!
    \| w_k - x_{k+1} \|^2 + 2 \la_k^{-1} \big[ \norm{D_k}^2 +
    \ipl D_k , F(x^*) - F(w_k) - A_k \, (x^* - w_k)  \ipr \big]
\\
& \!\!\!\! \geq & \!\!\!\!
    \| w_k - x_{k+1} \|^2 + 2 \la_k^{-1} \norm{D_k} \Big[
    \norm{D_k} - \eta \norm{y - F(w_k)} \Big]
\\
& \!\!\!\! \geq & \!\!\!\!
    \| w_k - x_{k+1} \|^2 + 2 \la_k^{-1} \norm{D_k} \,
    (q - \eta) \, \norm{y - F(w_k)} \, ,
\end{eqnarray*}
completing the proof.
\end{proof}

In the following proposition, we examine the boundedness of the
sequences $(x_k)$ and $(w_k)$ generated by Algorithm~\ref{alg:init-exact}.

\begin{propo} \label{prop:bound}
Let (A1) -- (A4) hold and $(x_k)$, $(w_k)$ be sequences generated by
Algorithm~\ref{alg:init-exact} (with $(\la_k)$ and $(\alpha_k)$ chosen
as in Steps~[1] and [2.3] respectively). If $(\alpha_k)$ satisfies
Assumption~\ref{ass:alpha} then $(x_k)$ and $(w_k)$ are contained
in $B_\rho(x_0)$.
\end{propo}
\begin{proof}
We present here a proof by induction, with inductive step stated
as follows: \smallskip

\centerline{Assume that $(w_k)_{k=0}^{l-1}$, $(x_k)_{k=0}^l \in
B_\rho(x_0)$, \ and conclude that $w_l$, $x_{l+1} \in B_\rho(x_0)$.}
\smallskip

\noindent For $l=1$, it holds $w_0 = x_0 \in B_\rho(x_0)$ and
$x_1 := w_0 + s_0$, where $\big( A^*_0 A_0 + \lambda_0 I \big)
s_0 = A^*_0 [y - F(w_0)]$. Thus, from Proposition~\ref{prop:gain},
(A3), (A4) follows $\norm{w_0 - x^\star} \geq \norm{x_1 - x^\star}$,
hence $x_1 \in B_\rho(x_0)$.
Now, the fact $x_0, x_1 \in B_\rho(x_0)$ together with
Assumption~\ref{ass:alpha} imply $w_1 \in B_\rho(x_0)$ (see
Remark\,\ref{rem:alpha})).
\smallskip

\noindent For $l > 1$, the inductive assumption ensures $x_{l-1}$,
$x_l \in B_\rho(x_0)$. Consequently, Assumption~\ref{ass:alpha}
implies $w_l \in B_\rho(x_0)$. Thus, arguing with Proposition~%
\ref{prop:gain}, (A3), (A4) and the inductive assumption we obtain
$\norm{w_l - x^\star} \geq \norm{x_{l+1} - x^\star}$.
Therefore, $x_{l+1} \in B_\rho(x_0)$, concluding the proof.
\end{proof}

In the upcoming proposition we discuss the summability of three series
related to inLM, a crucial element for proving a convergence theorem
(see Theorem~\ref{th:conv}).

\begin{propo}\label{prop:series}
Assume that $\alpha_{k}$ fulfills Assumption~\ref{ass:alpha}, that
(A1)-(A4) are fulfilled and that $(x_k)$ and $(w_{k})$ are generated
by Algorithm~\ref{alg:init-exact}. Then it holds that the limit
$\norm{x_k - x^*}$ exist for all solutions $x^*$ and that
\begin{align*}
\sum_{k=0}^{\infty} \norm{x_{k+1} - w_{k}}^2 < \infty \, ,
\quad
\sum_{k=0}^{\infty} \lambda_k^{-1} \norm{F(w_k) - y}^2 < \infty
\quad \text{and} \quad
\sum_{k=0}^{\infty} \lambda_k^{-1} \norm{F(x_k) - y}^2 < \infty .
\end{align*}
\end{propo}
\begin{proof}
From \eqref{eq:w} with $x = x^\star$ and Proposition~ \ref{prop:gain}
we conclude that%
\footnote{Notice that in Algorithm~\ref{alg:init-exact} we define $x_k$
for $k \geq 0$ and $\alpha_k$ for $k\geq 1$. For this proof we additionally
define $x_{-1} := x_0$ and $\alpha_0 := 0$; thus, \eqref{eq:w} holds
trivially for $k = 0$.}
\begin{multline*}
(1+\alpha_k) \norm{x_k -x^*}^2 - \alpha_k \norm{x_{k-1} -x^*}^2
+ \alpha_k (1+\alpha_k) \norm{x_k - x_{k-1}}^2 - \norm{x_{k+1} -x^*}^2
= \\
= \norm{w_k - x^*}^2 - \norm{x_{k+1} -x^*}^2
\ \geq \ \norm{w_k - x_{k+1}}^2 + 2 \la_k^{-1} (q-\eta) \,\norm{D_k}
       \,\norm{F(w_k) - y} \, , \ k \geq 0 \, .
\end{multline*}
Thus, defining $\varphi_k := \norm{x_k-x^*}^2$ \,and
\,$\eta_k := \alpha_k \norm{x_k - x_{k-1}}^2$, we obtain
\begin{multline} \label{eq:gamma}
\alpha_k (\varphi_k - \varphi_{k-1}) - (\varphi_{k+1} - \varphi_k)
+ (1+\alpha_k) \eta_k \ \geq \\
\ \geq \ \norm{x_{k+1} - w_k}^2 + 2 \la_k^{-1} (q-\eta) \,\norm{D_k}
     \,\norm{F(w_k) - y} \, , \ k \geq 0 \, .
\end{multline}
Since $\alpha_k<1$ we get from there that
\begin{align*}
\varphi_{k+1} - \varphi_{k} - \alpha_k(\varphi_k-\varphi_{k-1}) &
\leq (1+\alpha_k)\eta_k - \norm{w_{k}-x_{k+1}}^2
- \tfrac{2(q-\eta)}{\lambda_k}\norm{D_k}\norm{F(w_k)-y}\\
& \le 2\eta_{k}.
\end{align*}
We abbreviate $\zeta_k := \varphi_k-\varphi_{k-1}$ and write
$[\zeta_{k}]_{+}$ for the positive part and get
\begin{align*}
\zeta_{k+1}\leq \alpha_k[\zeta_{k}]_{+} + 2\eta_{k} 
\end{align*}
and hence with $\alpha_{k}\leq \alpha<1$
\begin{align*}
[\zeta_{k+1}]_+ \leq \alpha_k [\zeta_k]_+ + 2\eta_k \leq
\cdots\leq \alpha^k[\zeta_1]_+ + 2 \sum\limits_{j=0}^{k-1}\alpha^{j}\eta_{k-j}.
\end{align*}
We sum this inequality from $k=0,\dots,\infty$ and get 
\begin{align} \label{eq:sum-theta-k-+}
\sum_{k=0}^{\infty} [\zeta_{k+1}]_+
\leq \tfrac1{1-\alpha} [\zeta_1]_+
+ 2 \sum_{k=0}^{\infty} \sum_{j=0}^{k-1} \alpha^{j} \eta_{k-j} \, .
\end{align}
The latter sum can be calculated by swapping the order and
substitution: 
\begin{align*}
\sum_{k=0}^{\infty} \sum_{j=0}^{k-1} \alpha^{j} \eta_{k-j}
& = \sum_{j=0}^{\infty} \sum_{k=j+1}^{\infty} \alpha^{j} \eta_{k-j}
\\
& \sum_{j=0}^{\infty} \alpha^j \sum_{l=1}^{\infty} \eta_l
= \tfrac1{1-\alpha} \sum_{l=1}^{\infty} \eta_l \, .
\end{align*}
Thus~\eqref{eq:sum-theta-k-+} turns into 
\begin{align*}
\sum_{k=0}^{\infty} [\zeta_{k+1}]_+
\leq \tfrac1{1-\alpha} \left( [\zeta_1]_+
+ 2 \sum_{l=1}^{\infty} \eta_l \right) \, .
\end{align*}
The series on the right hand side is convergent by assumption
and hence $\sum_{k=0}^{\infty}[\zeta_k]_+ < \infty$.

Now we define $\gamma_k = \varphi_k - \sum_{j=1}^k [\zeta_j]_{+}$
which is bounded from below since $\varphi_k \geq 0$ and the series
is convergent. Moreover we have (recalling the definition of
$\zeta_{k}$)
\begin{align*}
\gamma_{k+1} & = \varphi_{k+1} - [\zeta_{k+1}]_{+}
- \sum_{j=1}^k [\zeta_j]_{+}
\\
& \leq \varphi_{k+1} - \varphi_{k+1} + \varphi_{k}
- \sum_{j=1}^k [\zeta_j]_{+} = \gamma_k \, .
\end{align*}
We see that $\gamma_k$ is non-increasing and bounded from below,
hence convergent. This implies that the limits
\begin{align*}
\lim_{k\to\infty} \varphi_k = \lim_{k\to\infty} \gamma_k
+ \sum_{j=1}^{\infty} [\zeta_j]_{+}
\end{align*}
all exist and since $\varphi_{k} = \norm{x_{k}-x^{*}}^2$ we get
that $\lim_{k\to\infty}\norm{x_k - x^*}$ exists.

Now we start at~\eqref{eq:gamma} again and write it as
\begin{align*}
2 \lambda_k^{-1} (q-\eta) \norm{D_k} \, \norm{F(w_k) - y}
+ \norm{x_{k+1} - w_k}^2 \leq
2 \eta_k + \varphi_k - \varphi_{k+1} + \alpha_k[\zeta_k]_+ \, .
\end{align*}
Using Lemma~\ref{lemma:dk-and-ineq} c) we get 
\begin{align*}
2 \lambda_k^{-1} q(q-\eta) \norm{F(w_k) - y}^2
+ \norm{x_{k+1} - w_k}^2 \leq
2 \eta_k + \varphi_k - \varphi_{k+1} + \alpha_k [\zeta_k]_+ \, .
\end{align*}
Summing from $k=0,\dots,\infty$ gives 
\begin{align*}
\sum_{k=0}^{\infty} \left(2 \lambda_k^{-1} q(q-\eta)
\norm{F(w_k) - y}^2 + \norm{x_{k+1} - w_k}^2 \right)
\leq \varphi_1 + \sum_{k=0}^{\infty} \left( \alpha_k [\zeta_k]_+
+ 2\eta_k \right)
\end{align*}
and since the right hand side of this inequality is bounded,
we get that $\sum_k \norm{x_{k+1} - w_k}^2 < \infty$ and
$\sum_k \lambda_k^{-1} \norm{F(w_k) - y}^2 < \infty$.
Lemma~\ref{lemma:dk-and-ineq} d) implies that the last series is
convergent as well.
\end{proof}

\subsection{A strong convergence result} \label{ssec:2.2}

In what follows we prove a (strong) convergence result for the inLM
method (Algorithm~\ref{alg:init-exact}) in the exact data case. To
prove this result we use two additional assumptions:
\smallskip

\noindent
(A5) \ There exists $\la_{max} > 0 $ s.t. $\la_k \leq \la_{max}$, for
    $k\geq 0$; \smallskip

\noindent
(A6) \ $(\alpha_k)$ is monotone non-increasing (see Step~[2.3] of
Algorithm~\ref{alg:init-exact}). \smallskip

\begin{theorem}[Convergence] \label{th:conv}
Let (A1) -- (A6) hold and $(x_k)$, $(w_k)$ be sequences generated by
Algorithm~\ref{alg:init-exact} (with $(\la_k)$ and $(\alpha_k)$ chosen
as in Steps~[1] and [2.3] respectively). Additionally, assume that
$(\alpha_k)$ complies with Assumption~\ref{ass:alpha}.
Then, either the sequences $(x_k)$, $(w_k$) stop after a finite
number $k_0 \in \N$ of steps (in this case it holds $x_{k_0+1} = w_{k_0}$
and $F(w_{k_0}) = y$), or there exists $\bar{x} \in B_\rho(x_0)$,
solution of $F(x) = y$, s.t. $\lim_k x_k = \lim_k w_k = \bar{x}$.
\end{theorem}
\begin{proof}
We consider two cases. \smallskip

\noindent \textbf{Case I:} \ $F(w_{k_0}) = y$ for some $k_0 \in \mathbb{N}$.
\\
In this case, the sequences $(x_k)$, $(w_k)$ read $x_0, \dots, x_{k_0+1}$
and $w_0, \dots, w_{k_0}$. Moreover, it holds $x_{k_0+1} = w_{k_0}$ and
$F(w_{k_0}) = y$ (see Remark~\ref{rem:station}). \smallskip

\noindent \textbf{Case II:} \ $F(w_k) \neq y$, for every $k \geq 0$.
\\
Notice that, in this case, the real sequence $\big(\norm{F(x_k) - y}\big)$
is strictly positive. Moreover, it follows from (A5) and Proposition~%
\ref{prop:series} (see second series) that $\lim_k \norm{F(w_k) - y} = 0$.
Therefore, there exists a strictly monotone increasing sequence
$(l_j) \in \N$ satisfying
\begin{equation} \label{eq:min}
\norm{F(w_{l_j}) - y} \ \leq \ \norm{F(w_k) - y} \, ,
\quad \mbox{for} \quad k = 0, \dots,  l_j \, .
\end{equation}
Notice that, given $k > 0$ and $z \in B_\rho(x_0)$, it holds
\begin{eqnarray*}
\|w_k  -  z\|^2 &\!\!\!\! - &\!\!\! \|x_{k+1} - z\|^2 =
    - \norm{x_{k+1} - w_k}^2 - 2 \ipl x_{k+1} - w_k , w_k - z \ipr
      \nonumber \\
&\!\!\!\!\! \leq &\!\!\!
      2 \ipl w_k - x_{k+1} , w_k - z \ipr
      \nonumber \\
&\!\!\!\!\! = &\!\!\!
      2 \ipl \la_k^{-1} A_k^*( F(w_k) + A_k (x_{k+1} - w_k) -y), w_k -z \ipr
      \nonumber \\
&\!\!\!\!\! = &\!\!\!
      2 \la_k^{-1} \ipl -F(x_{k+1}) + F(w_k) + A_k(x_{k+1} - w_k)
      + F(x_{k+1}) - y , A_k (w_k - z) \ipr
      \nonumber \\
&\!\!\!\!\! \leq &\!\!\!
      2 \la_k^{-1} \Big[ \eta \norm{F(x_{k+1}) - F(w_k)} 
      + \norm{F(x_{k+1}) - y} \Big] \, \norm{A_k (w_k - z)} \nonumber \\
&\!\!\!\!\! \leq &\!\!\!
      2 \la_k^{-1} (1+\eta) \Big[ \eta \norm{F(x_{k+1}) - F(w_k) \pm y} 
      + \norm{F(x_{k+1}) - y} \Big] \, \norm{F(w_k) - F(z) \pm y}
      \nonumber \\
&\!\!\!\!\! \leq &\!\!\!
      2 \la_k^{-1} (1+\eta)^2 \Big[ \norm{F(x_{k+1})-y}\norm{F(w_k)-y}
      + \norm{F(x_{k+1})-y}\norm{F(z)-y} \Big] \nonumber \\
&\!\!\!\!\!  &\!\!\!\!
      + 2 \la_k^{-1} \eta (1+\eta) \Big[ \norm{F(w_k) - y}^2
      + \norm{F(w_k) - y} \norm{F(z) - y} \Big]
\end{eqnarray*}
(in the second inequality we used (A2); in the third inequality we used
\cite[Eq.(11.7)]{EngHanNeu96}). Taking $z = w_{l_j}$ in the last inequality
and arguing with Lemma~\ref{lemma:dk-and-ineq}~(d) and \eqref{eq:min}
it follows
\begin{equation} \label{eq:muk}
\norm{w_k - w_{l_j}}^2 - \norm{x_{k+1} - w_{l_j}}^2
\ \leq \ 2 (1+\eta) \Big[ 2 \frac{(1+\eta)^2}{1-\eta} + 2 \eta \Big] 
\, \la_k^{-1} \norm{F(w_k) - y}^2 \ =: \ \mu_k \, ,
\end{equation}
for $k = 0, \dots, l_j$.
Next we estimate the second term on the left-hand-side of \eqref{eq:muk}.
Lemma~\ref{lemma:aux} (with $x = w_{l_j}$) implies
$$
\norm{w_{k+1} - w_{l_j}}^2 =
(1+\alpha_{k+1}) \norm{x_{k+1} - w_{l_j}}^2 - \alpha_{k+1}
\norm{x_k - w_{l_j}}^2 + \alpha_{k+1} (1+\alpha_{k+1})
\norm{x_{k+1} - x_k}^2 ,
$$
for $k = 0, 1, \dots$; from where we conclude that
\begin{equation} \label{eq:xlj}
0\leq \norm{x_{k+1} - w_{l_j}}^2 \ \leq \ \norm{w_{k+1} - w_{l_j}}^2
+ \alpha_{k+1}
\big( \norm{x_k - w_{l_j}}^2 - \norm{x_{k+1} - w_{l_j}}^2 \big)
\, , \ k \geq 0 \, .
\end{equation}
Now, combining \eqref{eq:muk} with \eqref{eq:xlj}, and arguing with (A6),
we obtain
\begin{eqnarray} \label{eq:telesc}
\norm{w_k - w_{l_j}}^2 - \norm{w_{k+1} - w_{l_j}}^2
& \leq &
  \alpha_{k+1} \big( \norm{x_k -w_{l_j}}^2 - \norm{x_{k+1} -w_{l_j}}^2 \big)
  + \mu_k
  \nonumber \\
& \leq &
  \alpha_k \norm{x_k - w_{l_j}}^2 - \alpha_{k+1} \norm{x_{k+1} -w_{l_j}}^2
  + \mu_k \, , \label{eq:txlj}
\end{eqnarray}
for $k = 0, \dots, l_j$. Let $0 \leq m \leq l_j$. Adding up
\eqref{eq:telesc} for $k = m, \ldots, l_j-1$ follows
$$
\norm{w_m - w_{l_j}}^2 - \norm{w_{l_j} - w_{l_j}}^2
\ \leq \ \alpha_m \norm{x_m - w_{l_j}}^2
       - \alpha_{l_j} \norm{x_{l_j} - w_{l_j}}^2
       + \summ_{k=m}^{l_j-1} \mu_k \, ,
$$
from where we derive that, for any $\varepsilon > 0$ it holds
\begin{eqnarray*}
\norm{w_m - w_{l_j}}^2
& \leq &
  \alpha_m \norm{x_m - w_{l_j} \pm w_m}^2 + \summ_{k=m}^{l_j-1} \mu_k
  \nonumber \\
& \leq &
  \alpha_m \big( (1+\varepsilon) \norm{x_m - w_m}^2 + (1+\tfrac1\varepsilon)\norm{w_m - w_{l_j}}^2 \big)
  + \summ_{k=m}^{l_j} \mu_k
  \nonumber\\
& \leq &
  (1+\varepsilon) \summ_{k=m}^{l_j} \alpha_k \norm{x_k - w_k}^2
  + (1+\tfrac1\varepsilon)\alpha_m \norm{w_m - w_{l_j}}^2 + \summ_{k=m}^{l_j} \mu_k \, .
\end{eqnarray*}
Consequently, whenever $m \leq l_j - 1$, it holds
$$
(1-\tfrac{\varepsilon+1}\varepsilon\alpha_m) \norm{w_m - w_{l_j}}^2 \ \leq \
(1+\varepsilon) \summ_{k=m}^{l_j} \alpha_k \norm{x_k - w_k}^2 +
\summ_{k=m}^{l_j} \mu_k \, .
$$
Now we choose $\varepsilon>0$ such that $\tfrac{\varepsilon+1}
\varepsilon\alpha_{0}<1$, define $\beta := (1-\tfrac{\varepsilon+1}
\varepsilon\alpha_0)^{-1}$, it follows from (A6) and \eqref{def:inLM-w}
\begin{eqnarray} \label{eq:number17}
\norm{w_m - w_{l_j}}^2 & \leq &
  2 \beta \summ_{k=m}^\infty \alpha_k \norm{x_k - w_k}^2
  + \beta \summ_{k=m}^\infty \mu_k
  \nonumber \\
& \leq &  2 \beta\alpha^3 \summ_{k=m}^\infty  \norm{x_k - x_{k-1}}^2
        + \beta \summ_{k=m}^\infty \mu_k \, , \ m < l_j
\end{eqnarray}
(notice that $\beta > 0$ and $\alpha_k \leq \alpha$ due to (A6)).

Notice that (A2) together with Proposition~\ref{prop:series}
guarantee the summability of both series $\sum_k \mu_k$ and
$\sum_k \norm{x_k - x_{k-1}}^2$.
Thus, defining $s_m := 2 \beta\alpha^3 \sum_{k\geq m}
\norm{x_k - x_{k-1}}^2 + \beta \sum_{k\geq m} \mu_k$, for
$m \in \N$, follows $s_m \to 0$ as $m \to \infty$.

Let $n > m$ be given. Choosing $l_j > n$, it follows from
\eqref{eq:number17}
$$
\norm{w_n - w_m} \ \leq \
\norm{w_n - w_{l_j}} + \norm{w_{l_j} - w_m} \ \leq \
\sqrt{s_n} + \sqrt{s_m} \ \leq \ 2 \sqrt{s_m} \, ,
$$
from where we conclude that $(w_k)$ is a Cauchy sequence.
Consequently, $(w_k)$ converges to some $\bar{x} \in X$.
From Proposition~\ref{prop:series} (see first series) it
follows $\lim_k x_k = \lim_k w_k = \bar{x}$.

It remais to prove that $\bar{x}$ is a solution of $F(x) = y$.
It suffices to verify that $\norm{F(w_k) - y} \to 0$ as $k\to\infty$.
This fact, however, is a consequence of Proposition~\ref{prop:series}
(see second series) together with Assumption~(A5).
\end{proof}

\subsection{Regularization properties} \label{ssec:2.3}

In this section we address the noisy data case, i.e.  $\delta > 0$,
and investigate regularization properties of the inertial
Levenberg-Marquardt method. For noisy data the inLM method reads is stated in Algorithm~\ref{alg:init-noise}.

\begin{algorithm}[h!]
\begin{center}
\fbox{\parbox{13.5cm}{
$[0]$ choose an initial guess $x_0 \in X$;
      \ set $w_0^\delta := x_0$; \ $k := 0$; \ flag := 'FALSE';
\medskip

$[1]$ choose $\tau > (\eta+1)(q-\eta)^{-1}$, \ $\alpha \in [0,1)$
      \ and \ $(\la_k)_{k\geq0} \in \R^+$;
\medskip

$[2]$ {\bf repeat}
\smallskip

\ \ \ \ \ \ \ {\bf if \ $\big( \norm{F(w_k^\delta) - y^\delta}
> \tau\delta \big)$ \ then}
\smallskip

\ \ \ \ \ \ \ \ \ \ $[2.1]$ $A_k^\delta := F'(w_k^\delta)$; \
          compute \,$s_k^\delta \in X$ as the solution of
\smallskip

\centerline{$\big( (A_k^\delta)^* A_k^\delta + \lambda_k I \big)
\, s_k^\delta \ = \ (A_k^\delta)^* \,
\big( y^\delta - F(w_k^\delta) \big)$;}
\smallskip

\ \ \ \ \ \ \ \ \ \ $[2.2]$ $x_{k+1}^\delta := w_k^\delta + s_k^\delta$;
\smallskip

\ \ \ \ \ \ \ \ \ \ $[2.3]$ $k := k+1$;
\smallskip

\ \ \ \ \ \ \ \ \ \ $[2.4]$ choose \,$\alpha_k^\delta \in [0,\alpha]$; \
          $w_k^\delta := x_k^\delta +
          \alpha_k^\delta (x_k^\delta - x_{k-1}^\delta)$;
\smallskip

\ \ \ \ \ \ \ {\bf else}
\smallskip

\ \ \ \ \ \ \ \ \ \ $[2.5]$ $s_k^\delta := 0$;
\ $x_{k+1}^\delta := w_k^\delta$; \ $k^* := k$; \ flag := 'TRUE';
\smallskip

\ \ \ \ \ \ \ {\bf end if}
\smallskip

\ \ \ \ \,{\bf until} \ \big(flag = 'TRUE'\big)
\smallskip
} }
\end{center} \vskip-0.5cm
\caption{Inertial Levenberg-Marquardt method in the noisy data case.}
\label{alg:init-noise}
\end{algorithm}

\begin{remark}[Comments regarding Algorithm~\ref{alg:init-noise}]
\label{rem:alg2} \mbox{}

\noindent The {\em discrepancy principle} is used as stopping criterion in
Algorithm~\ref{alg:init-noise}. i.e. the loop in Step~[2] terminates at step
$k^* = k^*(\delta, y^\delta)$ s.t. $k^* := \min\{ k\in\N$;
$\norm{F(w_k^\delta) - y^\delta} \leq \tau\delta \}$, where $\tau > 1$.

Note that Algorithm~\ref{alg:init-noise} generate sequences
$(x_k^\delta)_{k=0}^{k^*+1}$ and $(w_k^\delta)_{k=0}^{k^*}$. The finiteness
of the stopping index $k^*$ in Step~[2.5] is addressed in Proposition~%
\ref{prop:kstar}.

For each $0 \leq k \leq k^*$, define $D_k^\delta := F(w_k^\delta) +
A_k^\delta (x_{k+1}^\delta-w_k^\delta) - y^\delta$.
It is straightforward to verify that the results stated in Lemma~%
\ref{lemma:aux} and Lemma~\ref{lemma:dk-and-ineq} remain valid in
the noisy data case (the corresponding proofs are analogous and will
be omitted).

Additionally, if the sequence of inertial paramenters $(\alpha_k^\delta)$
in Algorithm~\ref{alg:init-noise} is chosen in accordance with
\begin{equation} \label{def:alphakdelta}
\alpha_k^\delta :=
  \left\{ \begin{array}{cl}
    \min \left\{ \dfrac{\theta_k}{\norm{x_k^\delta-x_{k-1}^\delta}^2} ,
    \dfrac{\min \big\{ \theta_k, \ \rho-\norm{x_k^\delta-x_0} \big\} }
    {\norm{x_k^\delta-x_{k-1}^\delta}} , \alpha \right\}
    & \!\!\!\!\!\! , \ {\rm if } \ \norm{x_k^\delta - x_{k-1}^\delta} > 0 \\[2ex]
    0     & \!\!\!\!\!\! , \ {\rm otherwise}
  \end{array} \right.
\end{equation}
(where $(\theta_k)$ is chosen as in Assumption~\ref{ass:alpha}),
then Remark~\ref{rem:alpha}~(a) holds true for $k = 1, \dots, k^*$.
\end{remark}

In the sequel we extend the ``gain estimate'' in Proposition~%
\ref{prop:gain} to the noisy data case.

\begin{propo} \label{prop:gainN}
Let (A1) -- (A4) hold and $(x^\delta_k)$, $(w^\delta_k)$ be sequences
generated by Algorithm~\ref{alg:init-noise} (with $(\la_k)$ and
$(\alpha_k^\delta)$ chosen as in Steps~[1] and [2.4] respectively). If
$w^\delta_k \in B_\rho(x_0)$ for some $0 \leq k \leq k^*$, then
$$
\|w^\delta_k - x^*\|^2 - \|x^\delta_{k+1} - x^*\|^2 \geq
\|w^\delta_k - x^\delta_{k+1}\|^2 + 2\la_k^{-1} \|D^\delta_k\|
\big[ (q-\eta) \|y^\delta - F(w^\delta_k) \| - (\eta+1) \delta \big] ,
$$
for any $x^* \in B_\rho(x_0)$ solution of $F(x) = y$.
\end{propo}
\begin{proof}
Since Lemma~\ref{lemma:dk-and-ineq} remains valid in the noise data
case (see Remark~\ref{rem:alg2}), we make a similar argument as in
the proof of Proposition~\ref{prop:gain} to establish that
\begin{eqnarray*}
\| w^\delta_k &\!\!\!\!\! - &\!\!\!\!\! x^*\|^2 - \|x^\delta_{k+1}-x^*\|^2
\ = \ \|w^\delta_k - x^\delta_{k+1}\|^2
    + 2\big\ipl w^\delta_k - x^\delta_{k+1} ,x^\delta_{k+1} - x^* \big\ipr
    \nonumber\\
& = &
  \| w^\delta_k - x^\delta_{k+1} \|^2 + 2\la_k^{-1}
  \big\ipl (A^\delta_k)^* \big[ F(w^\delta_k) + A^\delta_k(x^\delta_{k+1}
  - w^\delta_k) - y^\delta \big], x^\delta_{k+1} - x^* \big\ipr
  \nonumber\\
& = &
  \|w^\delta_k - x^\delta_{k+1}\|^2 + 2\la_k^{-1}
  \big\ipl D^\delta_k , A^\delta_k (x^\delta_{k+1} - x^*) \big\ipr
  \nonumber\\
& = &
  \|w^\delta_k - x^\delta_{k+1}\|^2 + 2\la_k^{-1}
  \big\ipl D^\delta_k , A^\delta_k ({x}^\delta_{k+1} - x^*)
  \pm A^\delta_k w^\delta_k \pm F(w^\delta_k) \pm y^\delta \big\ipr
  \nonumber\\
& = & \|w^\delta_k - x^\delta_{k+1}\|^2 + 2\la_k^{-1}
  \big[ \norm{D^\delta_k}^2 + \big\ipl D^\delta_k , A^\delta_k
  (w^\delta_k - x^*) - F(w^\delta_k) + y^\delta \big\ipr \big]
  \nonumber\\
& = & \|w^\delta_k - x^\delta_{k+1}\|^2 + 2\la_k^{-1}
  \big[ \norm{D^\delta_k}^2 + \big\ipl D^\delta_k, F(x^*) - F(w^\delta_k)
  - A^\delta_k (x^* - w^\delta_k) + y^\delta - y \big\ipr \big]
  \nonumber\\
& \geq & \|w^\delta_k - x^\delta_{k+1}\|^2 + 2\la_k^{-1}
  \big[ \norm{D^\delta_k}^2 - \eta\norm{D_k} \norm{y - F(w^\delta_k)}
  - \norm{D^\delta_k} \delta \big]
  \nonumber\\
& \geq & \|w^\delta_k - x^\delta_{k+1}\|^2 + 2\la_k^{-1}
  \norm{D^\delta_k} \big[ \norm{D^\delta_k} - \eta
  \norm{y^\delta - F(w^\delta_k)} - (\eta+1) \delta \big]
  \nonumber\\
& \geq & \|w^\delta_k - x^\delta_{k+1}\|^2 + 2\la_k^{-1}
  \norm{D^\delta_k} \big[ (q-\eta) \norm{y^\delta - F(w^\delta_k)}
  - (\eta+1) \delta \big] ,
\end{eqnarray*}
completing the proof.
\end{proof}

\begin{corol} \label{cor:boundN}
Due to the choice $\tau > (\eta+1)(q-\eta)^{-1}$ in Step~[1] of
Algorithm~\ref{alg:init-noise}, it follows from Proposition~%
\ref{prop:gainN}
\smallskip

\centerline{$\| w^\delta_k - x^\star\|^2 - \| x^\delta_{k+1} - x^\star \|^2
\ \geq \
\| w^\delta_k - x^\delta_{k+1} \|^2 + 2\la_k^{-1}
\norm{D^\delta_k} \big[ (q-\eta) \tau\delta - (\eta+1) \delta \big]
\ \geq \ 0$,}
\medskip

\noindent
for $k=0, \dots k^*$. Consequently, under the assumptions of
Proposition~\ref{prop:gainN}, if $(\alpha_k^\delta)$ satisfies 
\eqref{def:alphakdelta} then $x_k^\delta$, $w_k^\delta \in B_\rho(x_0)$
for $k=0, \dots k^*$ (the proof of this assertion follows the lines
of the proof of Proposition~\ref{prop:bound} and is omitted).
\end{corol}

In the sequel we address the finiteness of the stopping index $k^*$
as defined in Step~[2.5] of Algorithm~\ref{alg:init-noise}.

\begin{propo} \label{prop:kstar}
Let (A1) -- (A4) hold and $(x_k^\delta)$, $(w_k^\delta)$ be sequences
generated by Algorithm~\ref{alg:init-noise} (with $(\la_k)$ and
$(\alpha_k^\delta)$ chosen as in Steps~[1] and [2.4] respectively).
Assume that $(\alpha_k^\delta)$ satisfies  \eqref{def:alphakdelta}.
If $\sum_k \la_k^{-1} = \infty$ the stopping index $k^*$ defined
in Step~[2.5] is finite. Additionaly, if $\la_k \leq \la_{\max}$ then

\centerline{$k^* \ \leq \ \la_{\max} \,
\Big( 2 q \tau\delta^2 \big[ (q-\eta)\tau - (\eta+1) \big] \Big)^{-1}
\Big[\rho^2 + 2 \sum_k \theta_k \Big]$.}
\end{propo}
\begin{proof}
Recall that Lemma~\ref{lemma:aux} and Lemma~\ref{lemma:dk-and-ineq}
remain valid in the noisy data case (see Remark~\ref{rem:alg2}).
We claim that, if $k^*$ is not finite the sequence of partial sums
$(\sigma_n)$ defined by $\sigma_n := \sum_{k=0}^n \alpha_k^\delta
\big( \norm{x_{k-1}^\delta -x^\star}^2 - \norm{x_k^\delta -x^\star}^2
\big)$ is bounded (here $x_{-1}^\delta := x_0^\delta$).
Indeed, from Lemma~\ref{lemma:aux} (with
$x = x^\star$) and Proposition~\ref{prop:gainN} follow
\\[-4ex]
\begin{multline*}
(1+\alpha_k^\delta) \norm{x_k^\delta -x^\star}^2
- \alpha_k^\delta \norm{x_{k-1}^\delta -x^\star}^2
+ \alpha_k^\delta (1+\alpha_k^\delta)
\norm{x_k^\delta - x_{k-1}^\delta}^2 - \norm{x_{k+1} -x^\star}^2
\ = \\
= \ \norm{w_k^\delta - x^\star}^2 - \norm{x_{k+1}^\delta -x^\star}^2
\, \geq \, 0 \, , \ k \geq 0 \, .
\end{multline*}
\\[-3ex]
Thus, $\alpha_k^\delta \big( \norm{x_k^\delta -x^\star}^2
- \norm{x_{k-1}^\delta -x^\star}^2 \big)
+ \norm{x_k^\delta -x^\star}^2
+ 2 \alpha_k^\delta \norm{x_k^\delta - x_{k-1}^\delta}^2
- \norm{x_{k+1}^\delta -x^\star}^2 \geq 0$.
Consequently, $\alpha_k^\delta \big( \norm{x_{k-1}^\delta -x^\star}^2
- \norm{x_k^\delta - x^\star}^2 \big) \leq
\norm{x_k^\delta - x^\star}^2 - \norm{x_{k+1}^\delta -x^\star}^2
+ 2 \alpha_k^\delta \norm{x_k^\delta - x_{k-1}^\delta}^2$, for
$k \geq 0$.
Adding the last inequality for $k = 0, \dots n$, and using
\eqref{def:alphakdelta} we obtain

\begin{equation} \label{eq:sigmaN}
\sigma_n \ \leq \ \norm{x_0^\delta -x^\star}^2
- \norm{x_{n+1}^\delta -x^\star}^2 + 2
\summ_{k=0}^n \alpha_k^\delta \norm{x_k^\delta - x_{k-1}^\delta}^2
\ \leq \ \rho^2 + 2 \summ_{k=0}^\infty \theta_k \, .
\end{equation}
The boundedness of sequence $(\sigma_n)$ follows from the
summability of $(\theta_k)$, proving our claim.

For each $0 \leq k \leq k^*$ we derive from Proposition~%
\ref{prop:gainN} and Lemma~\ref{lemma:dk-and-ineq}~(c)
\begin{eqnarray*}
2\la^{-1}_k q \tau \delta^2 [ (q-\eta)\tau - (\eta+1) ]
& \leq & 2 \la^{-1}_k \norm{D_k^\delta}
  \big[ (q-\eta) \norm{y^\delta - F(w_k^\delta)} - (\eta+1)\delta \big]
  \\
& \leq & \|w^\delta_k - x^\star\|^2 - \|x^\delta_{k+1} - x^\star\|^2 .
\end{eqnarray*}
This inequality, Lemma~\ref{lemma:aux} (with $x = x^\star$),
\eqref{def:alphakdelta} and Corollary~\ref{cor:boundN} allow us
to conclude that
\begin{eqnarray*}
2\la^{-1}_k q \tau \delta^2 [ (q-\eta)\tau - (\eta+1) ]
& \leq & \|x^\delta_k - x^\star\|^2 - \|x^\delta_{k+1} - x^\star\|^2
  + \alpha_k^\delta \|x^\delta_k - x^\star\|^2
  \\
& & - \ \alpha_k^\delta \|x^\delta_{k-1} - x^\star\|^2
  +  \alpha_k^\delta (1+\alpha_k^\delta) \|x^\delta_k - x^\delta_{k-1}\|^2
  \\
& \leq & \|x^\delta_k - x^\star\|^2 - \|x^\delta_{k+1} - x^\star\|^2
  + \alpha_k^\delta \big( \|x^\delta_k - x^\star\|^2
  - \|x^\delta_{k-1} - x^\star\|^2 \big) + 2\theta_k
\end{eqnarray*}
for $0 \leq k \leq k^*$. Summing up the last inequality for $k = 0,
\dots, n$ with $n \leq k^*$ gives us
\begin{equation} \label{eq:finitek}
2q \tau \delta^2 \big[ (q-\eta)\tau - (\eta+1) \big]
\summ_{k=0}^n \la^{-1}_k \ \leq \
\| x_0 - x^\star \|^2 + \sigma_n + 2\summ_{k=0}^n \theta_k \, .
\end{equation}

If $k^*$ is not finite, it follows from the boundedness of $(\sigma_n)$
and the summability of $(\theta_k)$ that the right hand side of
\eqref{eq:finitek} is bounded. However, this contradicts the
assumption $\sum_k \la_k^{-1} = \infty$. Thus, $k^*$ has to be
finite.
To prove last assertion, note that the additional
assumption $\la_k \leq \la_{\max}$ together with \eqref{eq:finitek}
and \eqref{eq:sigmaN} imply 
$2q \tau\delta^2 \big[(q-\eta)\tau - (\eta+1)\big] \la^{-1}_{\max}
\, k^* \leq 2 \rho^2 + 4 \sum_k \theta_k$, concluding the proof.
\end{proof}

In the sequel we present the main results of this section, namely a
stability result (see Theorem~\ref{th:stabil}) and a semi-convergence
result (see Theorem~\ref{th:semi-conv}).

\begin{theorem}[Stability] \label{th:stabil}
Let (A1) hold, $(\delta^j)$ be a sequence of positive numbers
converging to zero and $(y^{\delta^j})$ be a sequence of noisy data
satisfying $\norm{y^{\delta^j} - y} \leq \delta^j$, where $y \in$
Rg$(F)$.
For each $j \in \N$, let $(x_l^{\delta^j})_{l=0}^{k^*_j+1}$
and $(w_l^{\delta^j})_{l=0}^{k^*_j}$ be the corresponding sequences
generated by Algorithm~\ref{alg:init-noise}, with $(\la_k)$ and
$(\alpha_k^{\delta^j})$ chosen as in Steps~[1] and [2.4] respectively
(here $k^*_j$ represent the stopping indices defined in Step~[2.5]).
Additionaly, assume that $(\alpha_k^{\delta^j})$ complies with
\eqref{def:alphakdelta}.

Let $(x_k)$ and $(w_k)$ be the sequences generated by
Algorithm~\ref{alg:init-exact} with $(\alpha_k)$ satisfying
Assumption~\ref{ass:alpha}. For each $k \geq 0$ it holds
\begin{equation} \label{eq:stabil}
\lim_{j\to\infty} x_k^{\delta^j} = x_k \quad {\it and} \quad
\lim_{j\to\infty} w_k^{\delta^j} = w_k
\end{equation}
(in view of Remark~\ref{rem:station}, if $(x_k)$ and $(w_k)$
are finite then the first limit in \eqref{eq:stabil} holds for
$k=0, \dots, k_0+1$, while the second limit holds for $k=0, \dots,
k_0$).%
\footnote{In this case $x_{k_0+1} = w_{k_0}$ and $F(w_{k_0})=y$.}
\end{theorem}
\begin{proof}
We give an inductive proof. In what follows we adopt
the simplifying notation $A_k^j := F'(w_k^{\delta^j})$. 
Notice that $w_0^{\delta^j} = w_0 = x_0 = x_0^{\delta^j}$ for
all $j \in \N$. Thus, \eqref{eq:stabil} holds for $k = 0$.
Next, assume the existence of $(x_{l})_{l\leq k}$ and
$(w_{l})_{l\leq k}$ generated by Algorithm~\ref{alg:init-exact}
(and corresponding $(\alpha_l)_{l\leq k}$ satisfying Assumption~%
\ref{ass:alpha}) such that $\lim_j x_l^{\delta^j} = x_l$ and
$\lim_j w_l^{\delta^j} = w_l$, for $l = 0, \dots, k$.

If $F(w_k)=y$ then $k_0 := k$, $x_{k_0+1} = w_{k_0}$ and
\eqref{eq:stabil} holds only for a finite number of indexes (i.e.
$(x_k)$ and $(w_k)$ are finite).
If $F(w_k) \not= y$, it follows from Algorithms~\ref{alg:init-exact}
and~\ref{alg:init-noise} that \\[1ex]
\centerline{$s_k=( A_k^*A_k + \la_kI)^{-1} A_k^*(y-F(w_k))$ \ and \
$s^{\delta^j}_k = [ (A_k^j)^*A_k^j + \la_k I ]^{-1}
(A_k^j)^* (y^{\delta^j} - F(w^{\delta^j}_k))$.} \\[1ex]
Thus (A1), the assumption $\lim_j y^{\delta^j} = y$, the inductive
assumption $\lim_j w_k^{\delta^j} = w_k$, and the fact
$\min \big\{ \norm{A_k^*A_k + \la_kI} , \, \min_j
\norm{(A_k^j)^*A_k^j + \la_kI} \big\} \geq \la_k > 0$ allow
us to conclude that $\lim_j s_k^{\delta^j} = s_k$. Consequently,
\begin{equation} \label{eq:Lxkm1}
\lim_{j\to\infty} \, x_{k+1}^{\delta^j}
\ = \ \lim_{j\to\infty} \, (w_k^{\delta^j} + s_k^{\delta^j})
\ = \ w_k + s_k \ = \ x_{k+1} \, .
\end{equation}
At this point, two distinct cases must be considered: \\
{\bf Case I:} $x_{k+1} \not= x_k$. Choose $\alpha_{k+1}$ according
to Assumption~\ref{ass:alpha}. From $\lim_j x_k^{\delta^j} = x_k$,
\eqref{eq:Lxkm1} and \eqref{def:alphakdelta} it follows that
$\lim_j \alpha^{\delta^j}_{k+1} = \alpha_{k+1}$.
Defining $w_{k+1}$ as in Step~[2.3] of Algorithm~\ref{alg:init-exact},
we conclude that
$$
\lim_{j\to\infty} w_{k+1}^{\delta^j} \ = \
\lim_{j\to\infty} \big( x_{k+1}^{\delta^j}
  + \alpha_{k+1}^{\delta^j} (x_{k+1}^{\delta^j} - x_k^{\delta^j}) \big)
\ = \ x_{k+1} + \alpha_{k+1}(x_{k+1} - x_k) \ = \ w_{k+1} \, .
$$
{\bf Case II:} $x_{k+1} = x_k$. Assumption~\ref{ass:alpha} implies
$\alpha_{k+1} = 0$. Define $w_{k+1}$ as in Step~[2.1] of Algorithm~%
\ref{alg:init-exact} (in this case $w_{k+1} = x_{k+1}$).
From $\lim_j x_k^{\delta^j} = x_k$, $\sup_j \alpha_{k+1}^{\delta^j}
\leq \alpha$, \eqref{eq:Lxkm1}, and Step~[2.4] of Algorithm~%
\ref{alg:init-noise}, it follows that
$$
\lim_{j\to\infty} w_{k+1}^{\delta^j} \ = \
\lim_{j\to\infty} \big[ x_{k+1}^{\delta^j} + \alpha_{k+1}^{\delta^j}
     ( x_{k+1}^{\delta^j} - x_k^{\delta^j} ) \big]
     \ = \ x_{k+1} \ = \ w_{k+1} \, .
$$
Thus, in either case it holds $\lim_j w_{k+1}^{\delta^j} = w_{k+1}$,
concluding the inductive proof.
\end{proof}

\begin{theorem}[Semi-convergence] \label{th:semi-conv}
Let (A1) - (A6) hold, $(\delta^j)$ be a sequence of positive numbers
converging to zero, and $(y^{\delta^j})$ be a sequence of noisy data
satisfying $\norm{y^{\delta^j} - y} \leq \delta^j$, where $y \in$
Rg$(F)$.
For each $j\in\N$, let $(x_l^{\delta^j})_{l=0}^{k^*_j+1}$ and
$(w_l^{\delta^j})_{l=0}^{k^*_j}$ be sequences generated by
Algorithm~\ref{alg:init-noise}, with $(\la_l)$ and
$(\alpha_l^{\delta^j})$ chosen as in Steps~[1] and [2.4]
respectively, and $(\alpha_l^{\delta^j})_{l=0}^{k^*_j}$ complying
with \eqref{def:alphakdelta} (here, $k^*_j = k^*(\delta^j, y^j)$
is the stopping index defined in Step~[2.5]).

The sequence $(x_{k^*_j}^{\delta^j})_j$ converges strongly to some
$\bar{x} \in B_\rho(x_0)$, such that $F(\bar{x}) = y$.
\end{theorem}
\begin{proof}
Let $(x_k)$, $(w_k)$ be sequences generated by Algorithm~%
\ref{alg:init-exact} with exact data $y$ and $(\alpha_k)$ 
satisfying Assumption~\ref{ass:alpha}.
Since (A1) - (A6) hold, it follows from Theorem~\ref{th:conv}
the existence of $\bar{x} \in B_\rho(x_0)$, solution of $F(x) = y$,
s.t. $\lim_k x_k = \lim_k w_k = \bar{x}$.
We aim to prove that $\lim_j x_{k^*_j}^{\delta^j} = \bar{x}$.
It suffices to prove that every subsequence of
$(x_{k^*_j}^{\delta^j})_j$ has itself a subsequence converging
strongly to $\bar{x}$.

Denote an arbitrary subsequence of $(x_{k^*_j}^{\delta^j})_j$
again by $(x_{k^*_j}^{\delta^j})_j$, and represent by
$(k^*_j)_j \in \N$ the corresponding subsequence of indices.
Two cases are considered:
\medskip

\noindent {\bf Case~1.}
$(k^*_j)_j$ has a finite accumulation point. \\
In this case, we can extract a subsequence $(k^*_{j_m})$ of $(k^*_j)$
such that $k^*_{j_m} = n$, for some $n \in \N$ and all indices $j_m$.
Applying Theorem~\ref{th:stabil} to $(\delta^{j_m})$ and
$(y^{\delta^{j_m}})$, we conclude that
$w_{k^*_{j_m}}^{\delta^{j_m}} = w_n^{\delta^{j_m}} \to w_n$ and
$x_{k^*_{j_m}+1}^{\delta^{j_m}} = x_{n+1}^{\delta^{j_m}} \to x_{n+1}$,
as $j_m \to \infty$.
We claim that $F(w_n) = y$. Indeed, $\norm{F(w_n) - y} =
\lim_{j_m} \norm{F(w_n^{\delta^{j_m}})-y} \leq
\lim_{j_m} \big( \norm{F(w_n^{\delta^{j_m}}) - y^{\delta^{j_m}}}
+ \norm{y^{\delta^{j_m}} - y} \big) \leq
\lim_{j_m} (\tau+1) \,\delta^{j_m} = 0$.
I.e. in this case, the second
assertion of Theorem~\ref{th:conv} holds with $k_0 = n$. Thus,
$x_{n+1} = w_n = \bar{x}$.
\medskip

\noindent {\bf Case~2.} \
$(k^*_j)_j$ has no finite accumulation point. \\
In this case we can extract a monotone strictly increasing subsequence,
again denoted by $(k^*_j)_j$.
Take $\varepsilon > 0$. From Theorem~\ref{th:conv} follows the
existence of $K_1 = K_1(\varepsilon) \in \N$ such that
\begin{equation} \label{eq:eps3-1}
\norm{x_k - \bar{x}} \ < \
\T\frac13 \varepsilon \, , \ k \geq K_1 \, .
\end{equation}
Since $\sum_k \theta_k$ is finite (see Assumption~\ref{ass:alpha}),
there exists $K_2 = K_2(\varepsilon) \in \N$ such that
\begin{equation} \label{eq:eps3-2}
\summ_{k\geq K_2} \theta_k \ \leq \ \frac13 \varepsilon \, .
\end{equation}
Define $K = K(\varepsilon) := \max\{ K_1 , K_2 \}$.
Due to the monotonicity of $(k^*_j)_j$, there exists $J_1 \in \N$
such that $k^*_j > K$ for $j \geq J_1$.

Theorem~\ref{th:stabil} applied to the subsequences $(\delta^j)_j$
and $(y^{\delta^j})_j$, corresponding to $(k^*_j)_j$, implies the
existence of $J_2 \in \N$ s.t.
\begin{equation} \label{eq:eps3-3}
\norm{x_K^{\delta^j} - x_{K}} \ \leq \ \T\frac13 \varepsilon
\, , \ j > J_2 \, .
\end{equation}
Set $J := \max \{ J_1, J_2 \}$.
From Proposition~\ref{prop:gainN} (with $x^* = \bar{x}$) and Step~[2.4]
of Algorithm~\ref{alg:init-noise} follow
$\norm{x_{k+1}^{\delta^j} - \bar{x}} \leq \norm{w_k^{\delta^j} - \bar{x}}
\leq \norm{x_k^{\delta^j} - \bar{x}} + \alpha_k^{\delta^j}
\norm{x_k^{\delta^j} - x_{k-1}^{\delta^j}}$, for $j \geq J$ and
$k = 0, \dots, k^*_j - 1$.
\noindent Consequently,
\begin{equation} \label{eq:telescopic}
\norm{x_{k+1}^{\delta^j} - \bar{x}} -
\norm{x_k^{\delta^j} - \bar{x}} \ \leq \
\alpha_k^{\delta^j} \norm{x_k^{\delta^j} - x_{k-1}^{\delta^j}} \, ,
\end{equation}
for $j \geq J$ and $k = 0, \dots, k^*_j - 1$. 
Adding \eqref{eq:telescopic} for $k = K, \dots, k^*_j-1$ we obtain
\smallskip

\centerline{$\norm{x_{k^*_j}^{\delta^j} - \bar{x}} \ \leq \
\norm{x_K^{\delta^j} - \bar{x}} \ + \
\summ_{k=K}^{k^*_j-1} \alpha_k^{\delta^j}
\norm{x_k^{\delta^j} - x_{k-1}^{\delta^j}} \, , \ j \geq J$.}
\smallskip

\noindent
Thus, arguing with \eqref{def:alphakdelta}, together with
\eqref{eq:eps3-1}, \eqref{eq:eps3-2} and \eqref{eq:eps3-3}, we obtain
\smallskip

\centerline{$
\norm{x_{k^*_j}^{\delta^j} - \bar{x}}
\ \leq \
  \norm{x_K^{\delta^j} - \bar{x}} + \summ_{k=K}^{k^*_j-1} \theta_k
\ \leq \
  \norm{x_K^{\delta^j} - x_K} + \norm{x_K - \bar{x}}
  + \summ_{k\geq K} \theta_k
\ \leq \
  \frac13\varepsilon + \frac13\varepsilon + \frac13\varepsilon$,}
\smallskip

\noindent
for $j \geq J$. Repeating the above argument for $\varepsilon = 1,
\frac12, \frac13, \dots$ we  generate a sequence of indices
$j_1 < j_2 < j_3 < \dots$ such that
$\norm{x_{k^*_{j_m}}^{\delta^{j_m}} - \bar{x}} \leq
\frac1m$, for $m \in \N$.
This concludes Case~2, and completes the proof of the theorem.
\end{proof}

\section{Numerical experiments} \label{sec:numerics}

In this section two distinct ill-posed problems are used to investigate
the numerical efficiency of the inLM method.

\subsection{Parameter identification in an elliptic PDE} 

We aim to identify the coefficient $c \geq 0$ in the elliptic PDE on the unit square $\Omega = (0,1)^2$
with Dirichlet boundary condition
\begin{align}
\label{eqn:PDE_reco_problem_continuous}
-\Delta u+cu=g, \ \text{ in } \Omega
\quad\quad
u=\bar u \ \text{ on } \partial \Omega
\end{align}
from the knowledge of $u$ on the full domain $S\subset \Omega$. Here
the right-hand side $g \in L^2(\Omega)$ and boundary conditions
$\bar u\in H^{3/2}(\Omega)$ are known. This is a typical benchmark
inverse problem, see~\cite[Example 4.2]{HanNeuSch95}. If $u$ has no
zeroes in $\Omega$, then $c$ can be recovered explicitly by 
\begin{align}
\label{eqn:c_from_simple_division}
c = (g+\Delta u)/u.
\end{align}
However, for $u$ given with noise, this operation is expected to be
unstable as the application of the Laplacian is ill-conditioned. We
rearrange and discretize \eqref{eqn:PDE_reco_problem_continuous} with
a uniform grid of size $n\times n$ and use the standard five-point
stencil $\Delta_n$ with fineness $1/n$ in both dimensions as a
discretization of the Laplacian $\Delta$. We associate functions on
$\Omega$ with column vectors by assembling its values on the grid
and traversing row-wise. We state the discretized inverse problem
as asking for a reconstruction of $c^\delta\in\mathbb R^{n^2}$ from
a given noisy solution $u^\delta\in\mathbb R^{n^2}$ and a vector
$z\in\mathbb R^{n^2}$, which contains the discretized right-hand
side $g$ as well as boundary conditions $\bar u$ from
\eqref{eqn:PDE_reco_problem_continuous}, such that
\begin{align}
\label{eqn:PDE_reco_problem_discrete}
F(c^\delta) := \big( -\Delta_n + \mathrm{diag}(c^\delta) \big)^{-1}(z) = u^\delta,
\end{align}
where $\mathrm{diag}(c^\delta)$ denotes the diagonal matrix
with entries from $c^\delta$.
The mapping $F$ from \eqref{eqn:PDE_reco_problem_discrete} is
known to fulfill assumption (A2) locally.
For $g$ and $\bar u$ we set
\begin{align*}
g(x,y) &= 200 \cdot e^{-10\big(x-\tfrac{1}{2}\big)^2-10\big(y-\tfrac{1}{2}\big)^2}, \\
\bar u(x,y) &= 0, \qquad (x,y) \in \partial\Omega
\end{align*}
and $c^\dagger = c^\dagger_0 \ast \varphi$ with
\begin{align*}
c^\dagger_0(x,y) &= 
\begin{cases}
10, &  \text{if\ } \min\big( \ \sqrt{(x-0.25)^2+(y-0.5)^2}, \ \sqrt{(x-0.75)^2+(y-0.5)^2} \ \big) < \frac{1}{10}, \\
0, & \text{otherwise,}
\end{cases} \\
\varphi(x,y) &= \frac{1}{20\sqrt{\pi}} \cdot e^{-\frac{x^2+y^2}{200}}, \qquad x,y\in(0,1)
\end{align*}
and compute $u^\delta$ by a forward evaluation of the operator
$F$ from \eqref{eqn:PDE_reco_problem_discrete}. We choose $n=100$,
which means that all vectors have size $10^4$. We approximate the
linear solve in step [2.1] of Algorithm~\ref{alg:init-noise} by
two steps of the conjugate gradient method with initial value
zero for the $s$-variable.

We first examine the noiseless case. We depict our chosen $c^\dagger$
and the corresponding solution~$u^\dagger$ of the forward problem
in Figure~\ref{fig:noiseless_exact_solutions}. The vector~$u^\dagger$
is everywhere nonzero and hence we could recover $c^\dagger$ exactly
by division as in \eqref{eqn:c_from_simple_division}. The rightmost
subfigure in Figure~\ref{fig:noiseless_exact_solutions} shows that
this is perfectly possible in our case. Nevertheless, we test how
well the iterates $w^k$ of Algorithm~\ref{alg:init-exact} are able
to approximate the coefficient $c^\dagger$. We choose $w_0 = 0$
and consider $\alpha = k/10$ for $k=0,2,...,10$, where $\alpha=0$
corresponds to the non-accelerated Levenberg-Marquardt method and
$\alpha=1$ is not covered by our theory. 
To investigate convergence, we keep track of the residuals
$\|F(w_k)-u^\dagger\|$ and the distances $\|c_k-c^\dagger\|$.
The corresponding results can be seen in
Figure~\ref{fig:noiseless_res_and_distance}. We observe that all
methods converge, where convergence is faster for larger acceleration
parameters $\alpha$ except for $\alpha=1$.
In Figure~\ref{fig:noiseless_recos_10iter} we see that after 10
iterations, larger values of $\alpha$ proceed much faster in
reconstruction and $\alpha=1$ gives the best guess.
Figure~\ref{fig:noiseless_recos_500iter} shows that after 500
iterations the reconstructions look decent for $\alpha<1$, but
the peak shape is not fit for $\alpha=1$.

Next, we add $1\%$ of relative noise to the forward solution
$u^\dagger$, which yields a noisy vector~$u^\delta$ with no
visible difference from $u^\dagger$ (left subfigures in Figure
\ref{fig:noiseless_exact_solutions} and
Figure~\ref{fig:noisy_u_noisy_and_c_naive}). Here, the naive
calculation of $c^\delta$ by \eqref{eqn:c_from_simple_division}
fails drastically, as one sees from the right subfigure in
Figure~\ref{fig:noisy_u_noisy_and_c_naive}. We compute
reconstructions using Algorithm~\ref{alg:init-noise}, where
we again initialize by $w_0=0$ and choose $\alpha = k/10$ for
$k=0,2,...,10$. Figure~\ref{fig:noisy_res_and_distance} shows
the typical semi-convergence phenomenon. As one can see, the
closest distance to the true coefficient is achieved earlier
for larger values of~$\alpha$, which even includes $\alpha=1$.
We illustrate stopping by Morozov's discrepancy principle with
the horizontal line in the right subfigure of
Figure~\ref{fig:noisy_res_and_distance} and with bullet points
on the graphs, where we set $\tau=1$. From both
Figure~\ref{fig:noisy_res_and_distance} (left subfigure) and
from Figure~\ref{fig:noisy_recos_Morozov} one can see that
stopping happens too early even though we set $\tau=1$.
Indeed, we see that for $\alpha<1$ the residuals decay even
slightly below the absolute noise level, where the approached
residual value does not depend of the concrete value of
$\alpha<1$. After 100 iterations, the reconstruction looks
decent for $\alpha\leq 0.8$, but breaks down for $\alpha=1$,
see Figure~\ref{fig:noisy_recos_100_iter}. In
Figure~\ref{fig:noisy_recos_100_iter_closest} we see that the
reconstruction at the respective iterations where $w_k$ is
closest to $c^\dagger$ in Euclidean norm do not look different
for varying $\alpha$ (cf. Figure~\ref{fig:noisy_res_and_distance}).

 \begin{figure}[H]
 \begin{minipage}{.32\textwidth}
 \hspace{-1cm}
 \includegraphics[scale=0.4]{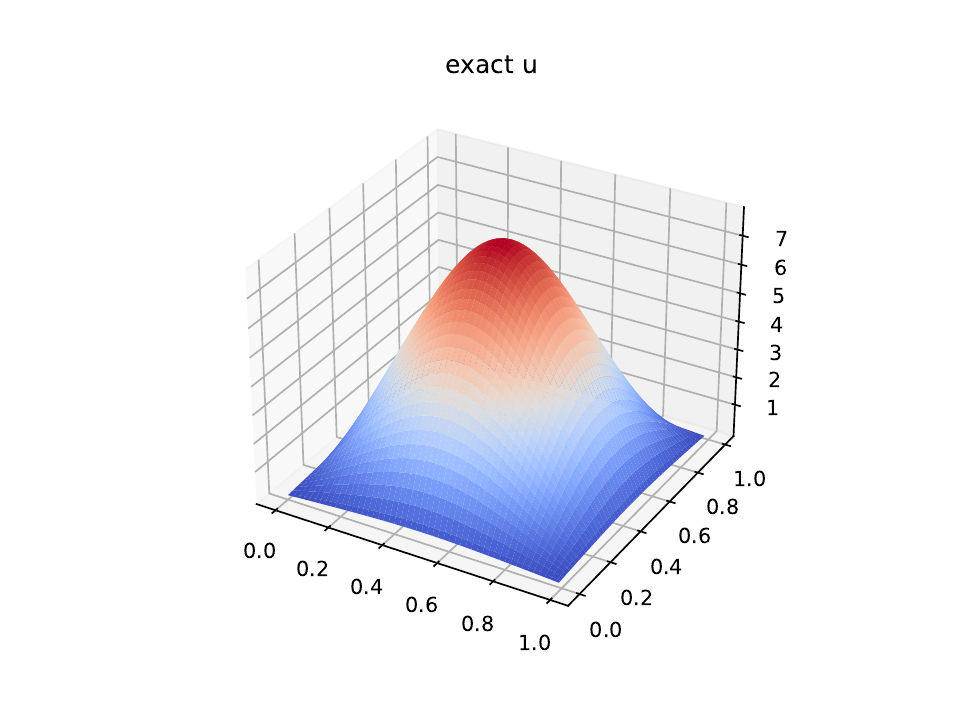}
 \end{minipage}
 \begin{minipage}{.32\textwidth}
 \hspace{-.75cm}
 \includegraphics[scale=0.4]{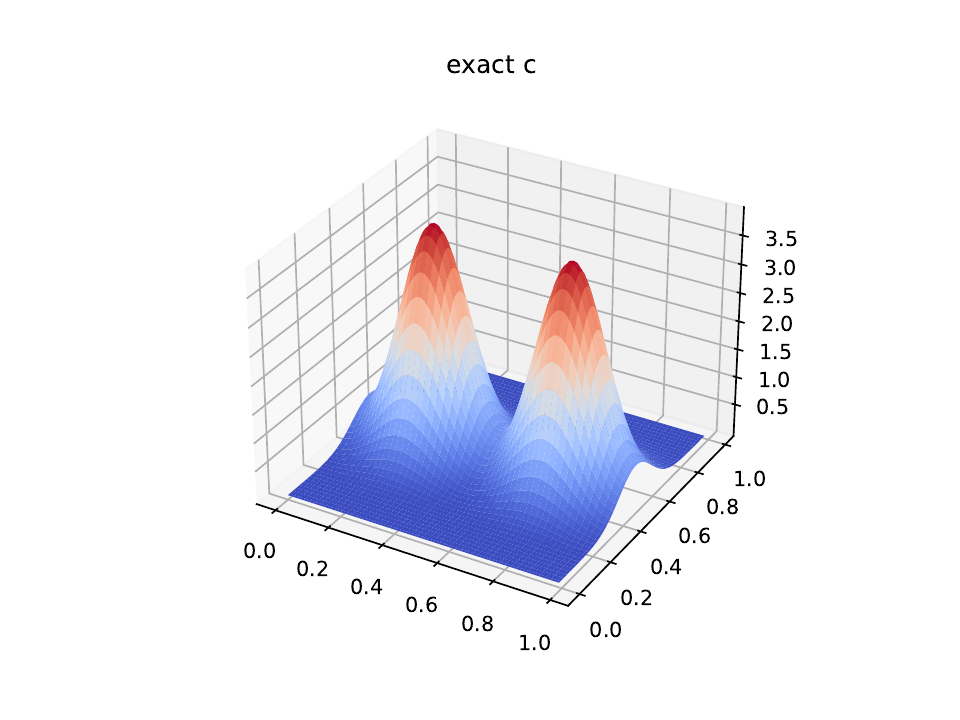}
 \end{minipage}
 \begin{minipage}{.32\textwidth}
 \hspace{-.5cm}
 \includegraphics[scale=0.4]{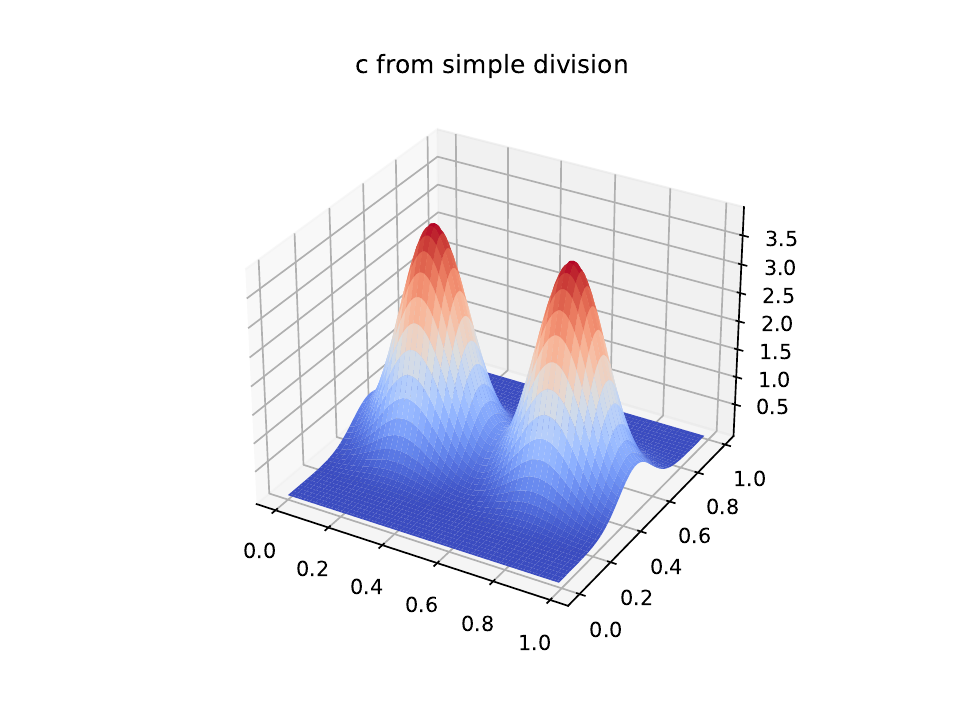}
 \end{minipage}
 \caption{Left to right: exact PDE solution $u^\dagger$ (exact
 solution to the forward problem), $c^\dagger$ (exact solution
 to the inverse problem), reconstruction of $c^\dagger$ by
 \eqref{eqn:c_from_simple_division}}
 \label{fig:noiseless_exact_solutions}
 \end{figure}
 
 \begin{figure}[H]
 \begin{minipage}{.45\textwidth}
 \includegraphics[scale=0.4]{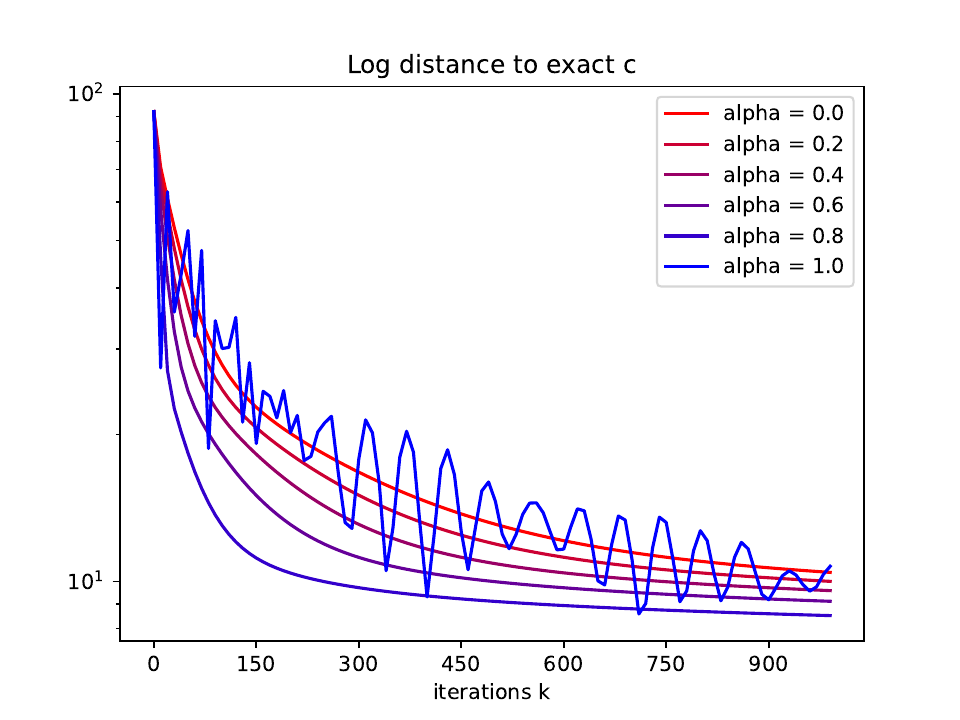}
 \end{minipage}
 \begin{minipage}{.45\textwidth}
 \includegraphics[scale=0.4]{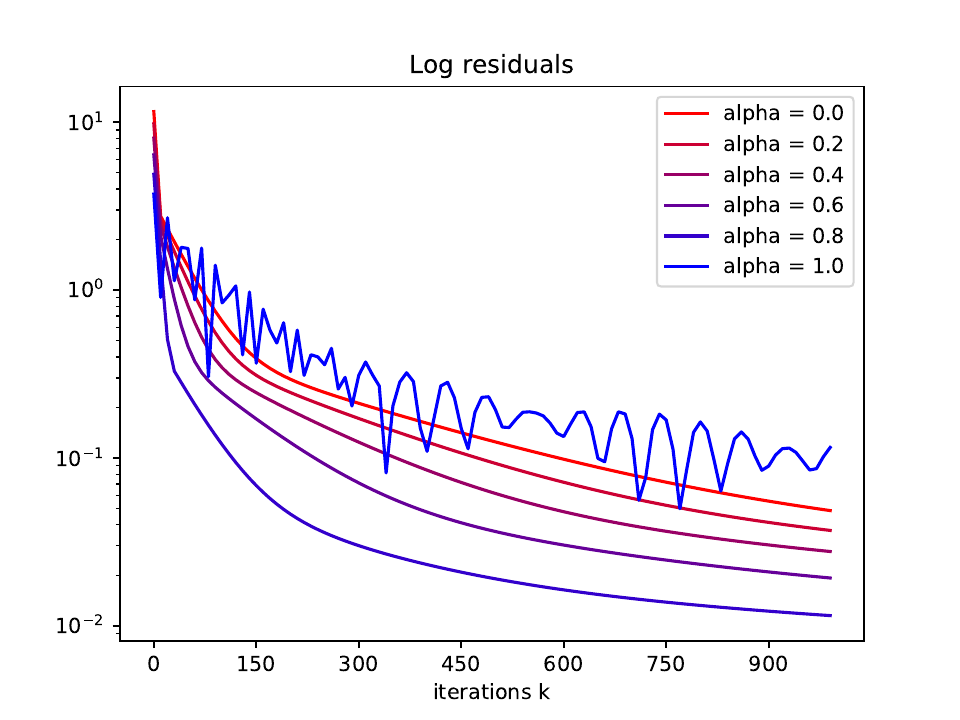}
 \end{minipage}
 \caption{Results without noise: errors for $\alpha_k\equiv\alpha\in
 \{0,0.2,...,1\}$ from red to blue color. Left: distances $\|c_k-c^\dagger\|_2$,
 right: residuals $\|F(c_k)-u^\dagger\|_2$}
 \label{fig:noiseless_res_and_distance}
 \end{figure}
 
 \begin{figure}[H]
 \begin{minipage}{.33\textwidth}
 \hspace{-1cm} 
 \includegraphics[scale=0.4]{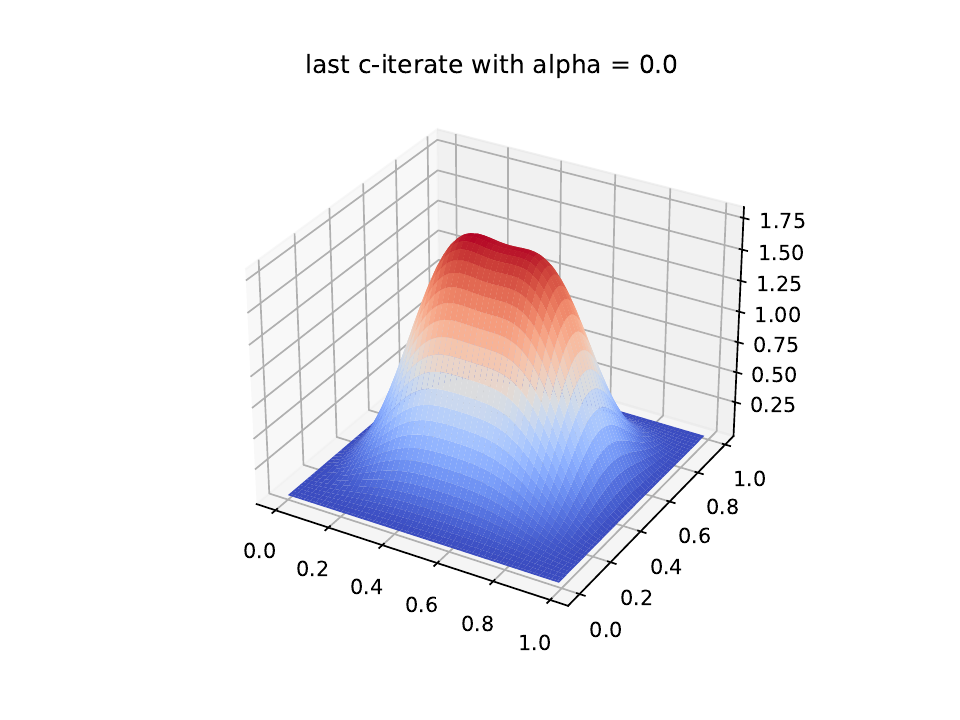}
 \end{minipage}
 \begin{minipage}{.33\textwidth}
 \hspace{-.75cm}
 \includegraphics[scale=0.4]{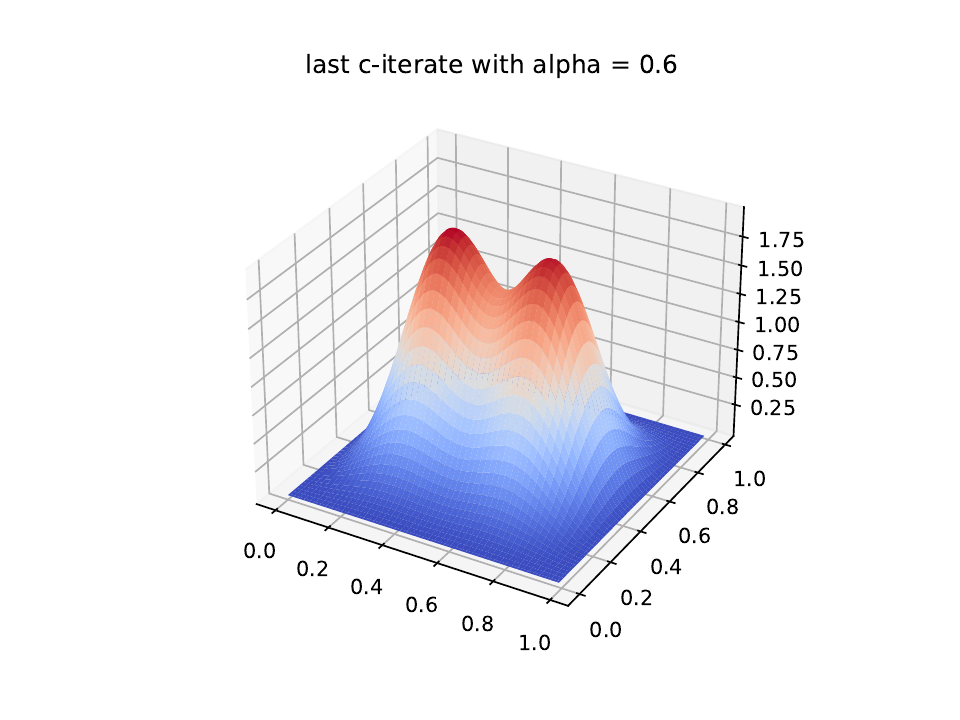}
 \end{minipage}
 \begin{minipage}{.32\textwidth}
 \hspace{-.5cm}
 \includegraphics[scale=0.4]{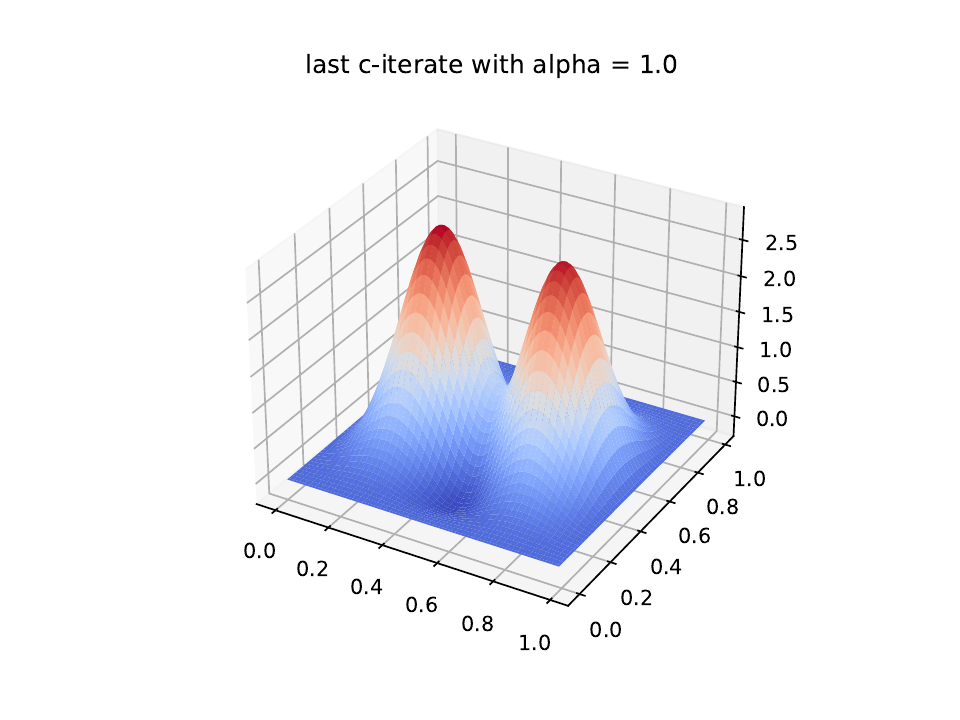}
 \end{minipage}
 \caption{Results without noise, left to right: iterates $w_{10}$
 for $\alpha=0$ (non-accelerated Levenberg-Marquardt method), $\alpha=0.6$
 and $\alpha=1$}
 \label{fig:noiseless_recos_10iter}
 \end{figure}
 
 \begin{figure}[H]
 \begin{minipage}{.32\textwidth}
 \hspace{-1cm}
 \includegraphics[scale=0.4]{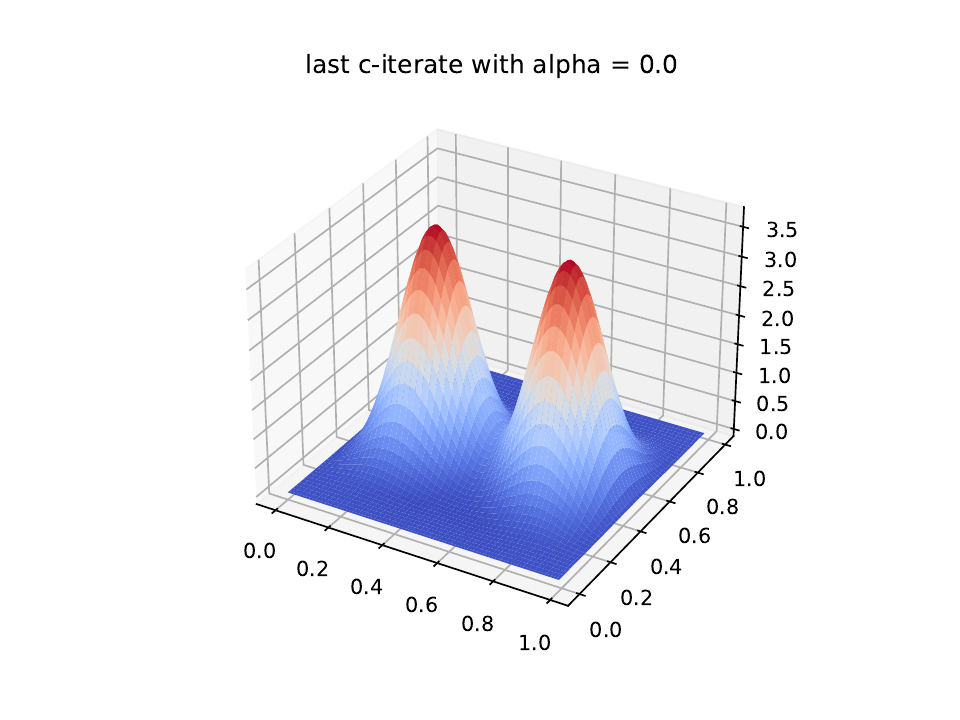}
 \end{minipage}
 \begin{minipage}{.32\textwidth}
 \hspace{-.75cm}
 \includegraphics[scale=0.4]{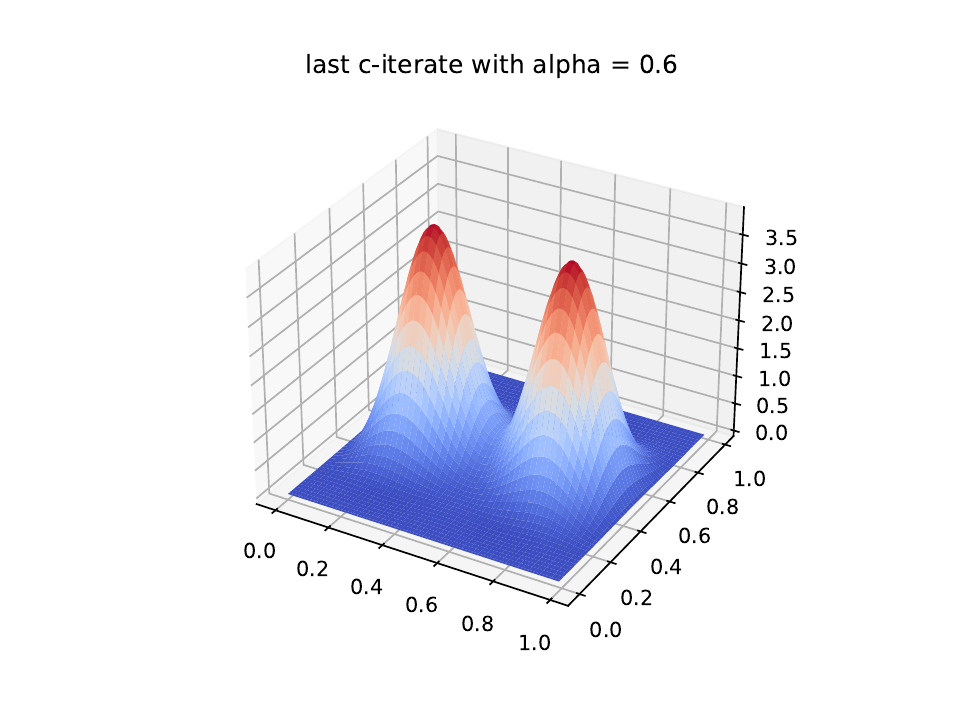}
 \end{minipage}
 \begin{minipage}{.32\textwidth}
 \hspace{-.5cm}
 \includegraphics[scale=0.4]{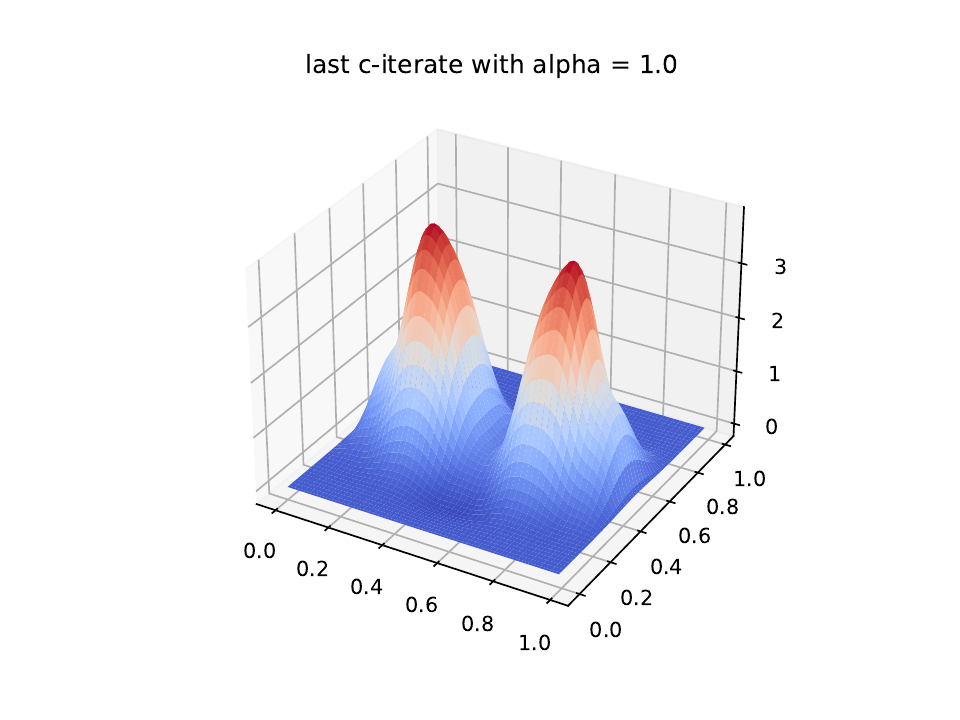}
 \end{minipage}
 \caption{Results without noise, left to right: iterates $w_{500}$
 for $\alpha=0$ (non-accelerated Levenberg-Marquardt method), $\alpha=0.6$
 and $\alpha=1$}
 \label{fig:noiseless_recos_500iter}
 \end{figure}
 
 \begin{figure}[H]
 \begin{minipage}{.45\textwidth}
 \includegraphics[scale=0.4]{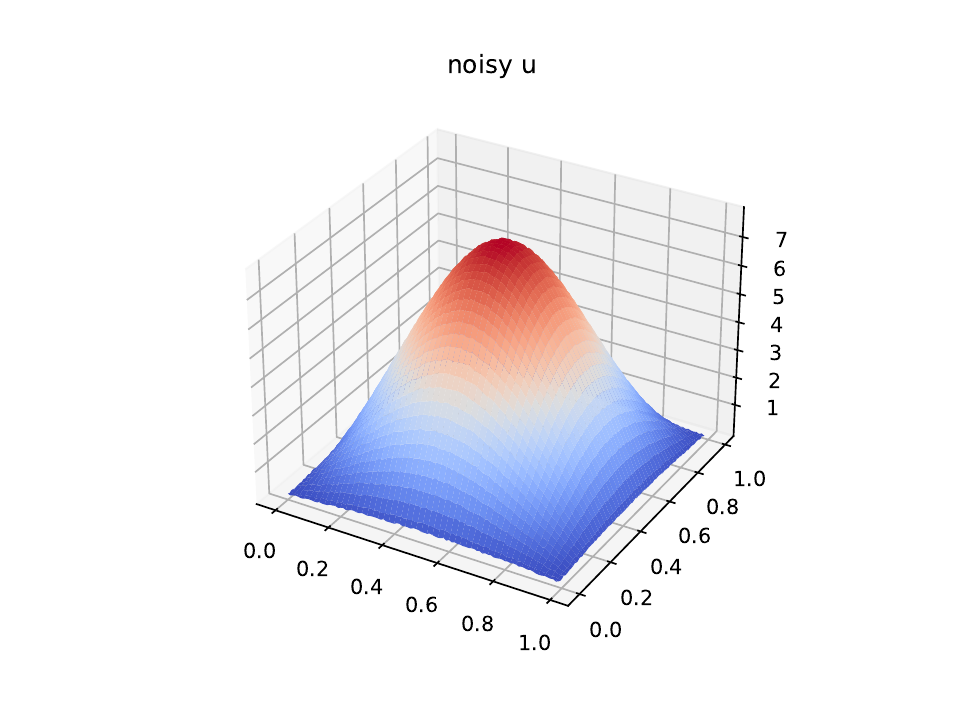}
 \end{minipage}
 \begin{minipage}{.45\textwidth}
 \includegraphics[scale=0.4]{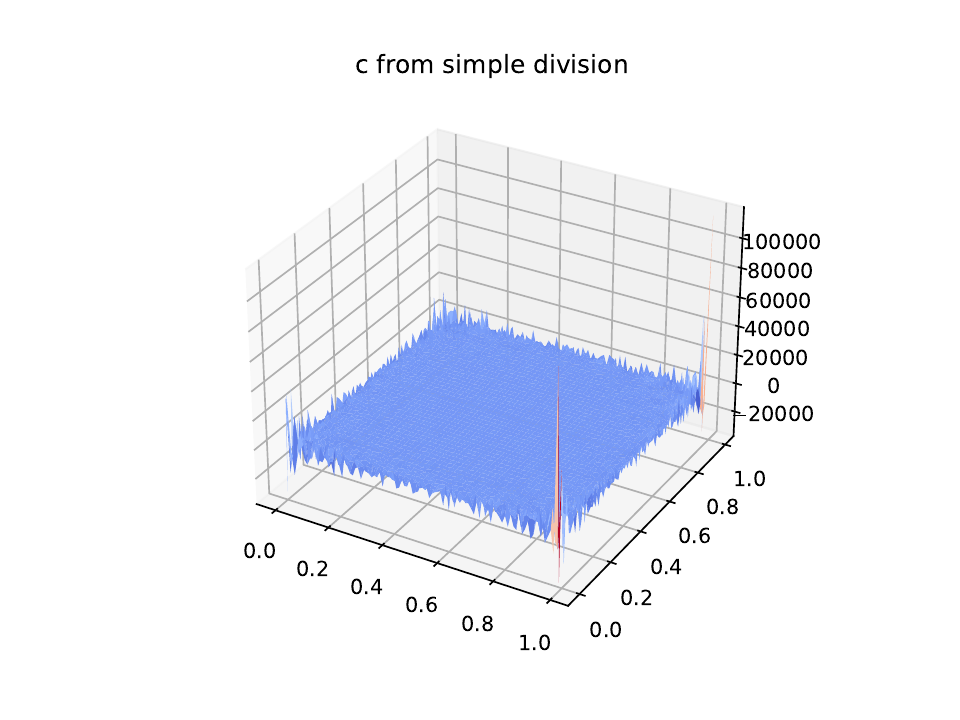}
 \end{minipage}
 \caption{Left: noisy PDE solution $u^\delta$, right: naive approximation
 of $c^\dagger$ by \eqref{eqn:c_from_simple_division}}
 \label{fig:noisy_u_noisy_and_c_naive}
 \end{figure}
 
 \begin{figure}[H]
 \begin{minipage}{.45\textwidth}
 \includegraphics[scale=0.4]{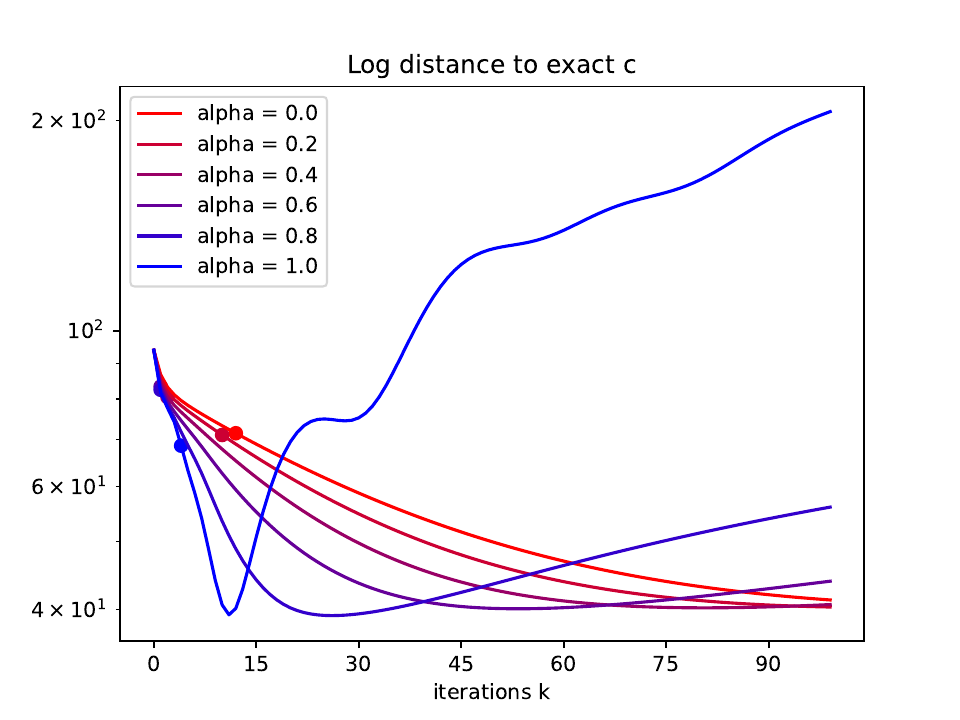}
 \end{minipage}
 \begin{minipage}{.45\textwidth}
 \includegraphics[scale=0.4]{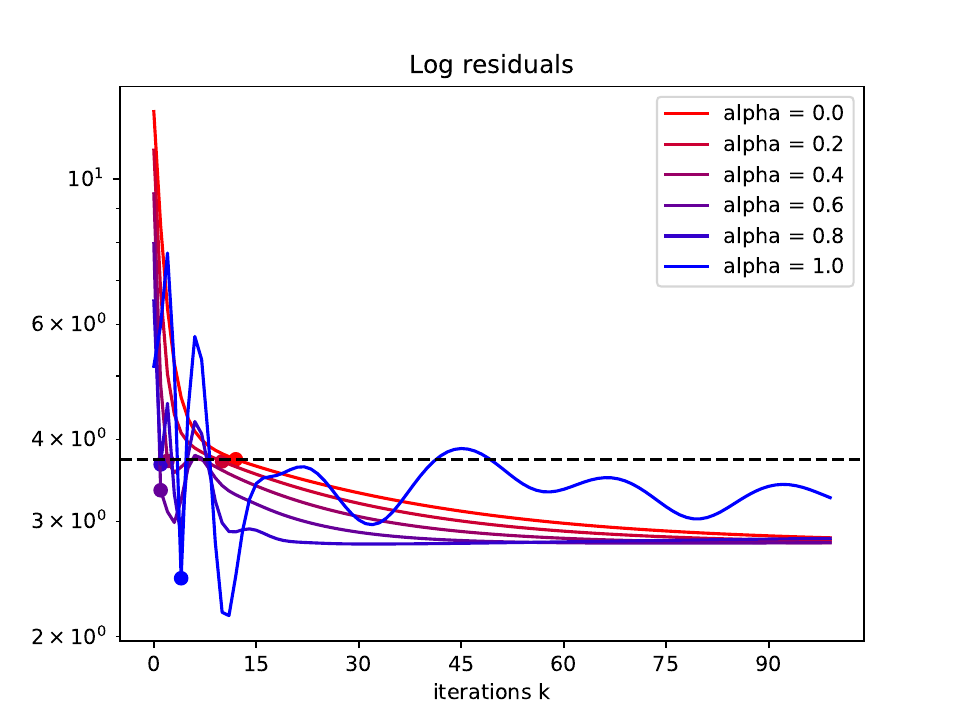}
 \end{minipage}
 \caption{Results with $1\%$ of relative noise, errors for
 $\alpha_k\equiv\alpha\in \{0,0.2,...,1\}$ from red to blue color. Left:
 distances $\|c_k-c^\dagger\|_2$, right: residuals $\|F(c_k)-u^\dagger\|_2$.
 Dots indicate stopping by Morozov's discrepancy principle with $\tau = 1$}
 \label{fig:noisy_res_and_distance}
 \end{figure}
 
 \begin{figure}[H]
 \begin{minipage}{.33\textwidth}
 \hspace{-1cm}
 \includegraphics[scale=0.4]{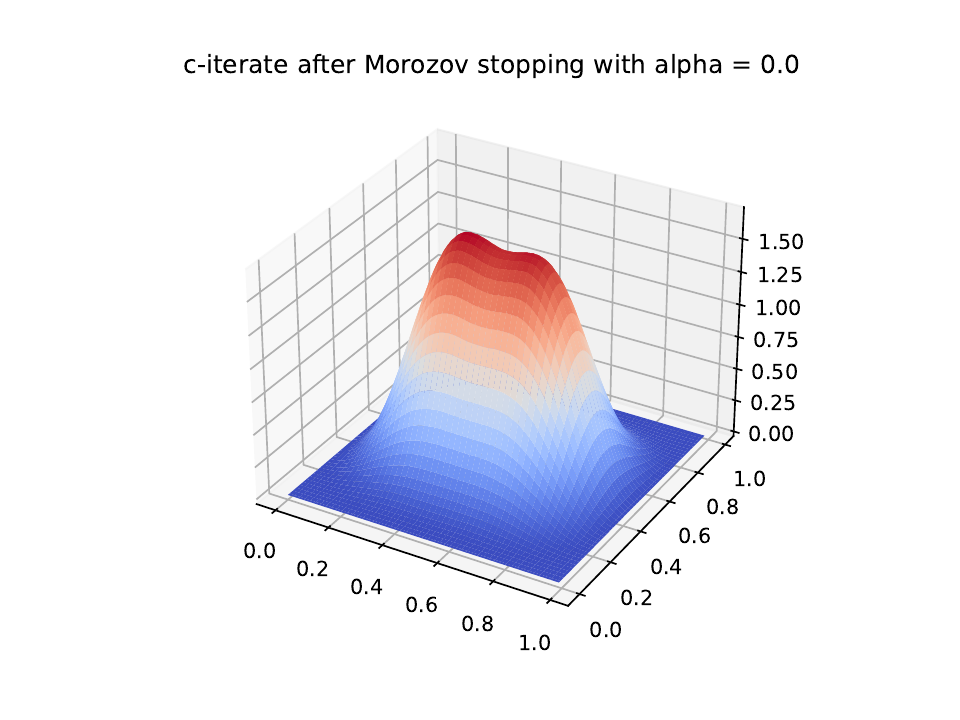}
 \end{minipage}
 \begin{minipage}{.33\textwidth}
 \hspace{-.75cm}
 \includegraphics[scale=0.4]{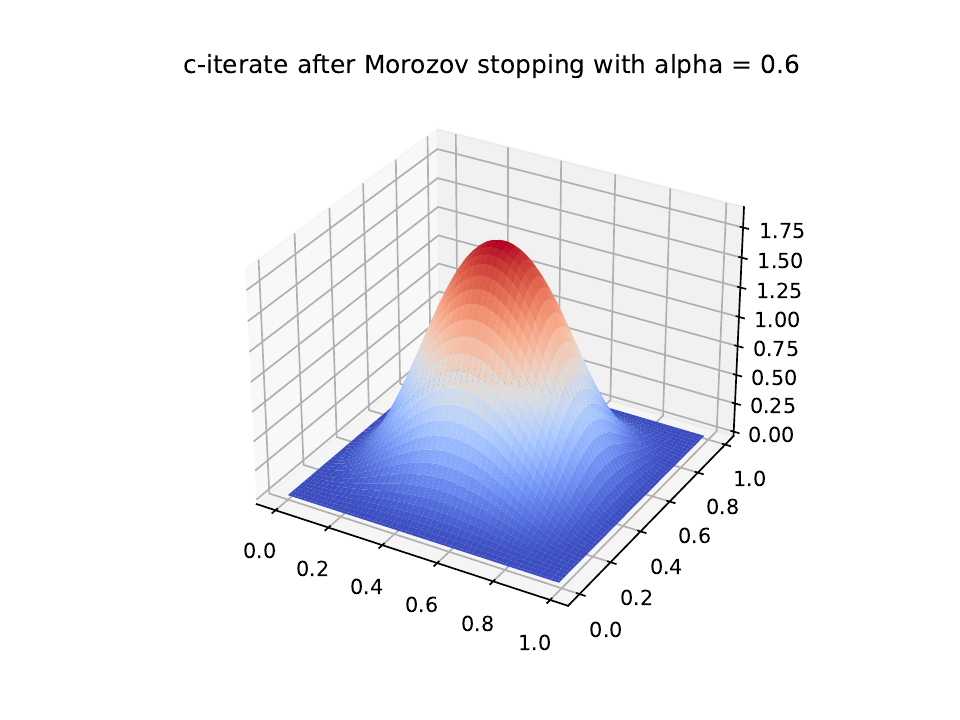}
 \end{minipage}
 \begin{minipage}{.32\textwidth}
 \hspace{-.5cm}
 \includegraphics[scale=0.4]{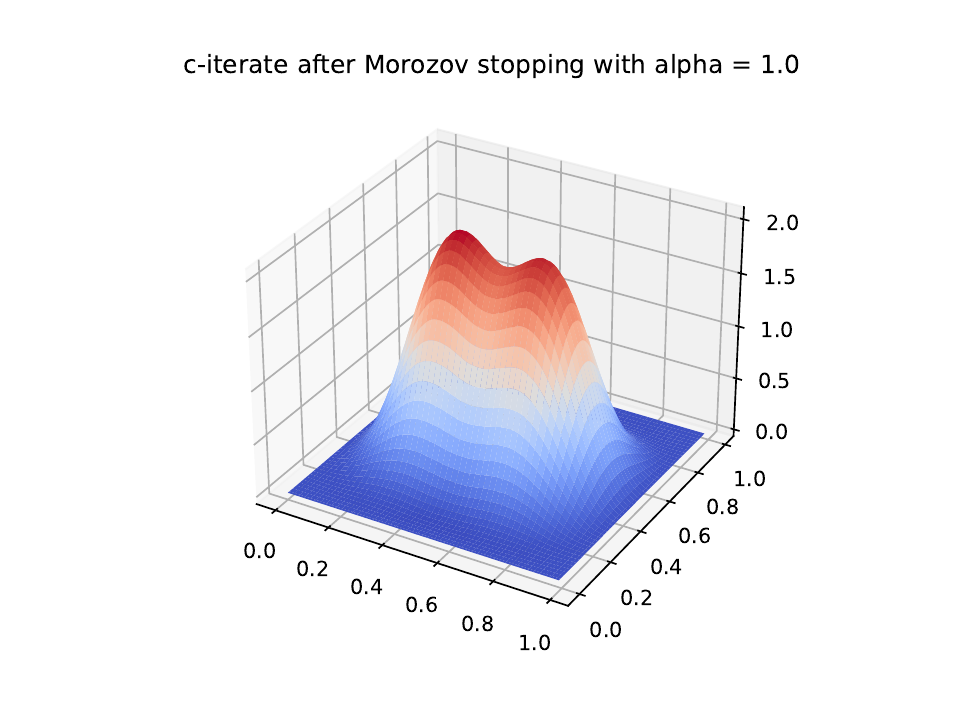}
 \end{minipage}
 \caption{Results with $1\%$ of relative noise, left to right: iterates
 $w_{k}^\delta$ stopped by Morozov's discrepancy principle with $\tau=1$ for
 $\alpha=0$ (non-accelerated Levenberg-Marquardt method), $\alpha=0.6$ and
 $\alpha=1.0$}
 \label{fig:noisy_recos_Morozov}
 \end{figure}
 
 \begin{figure}[H]
 \begin{minipage}{.32\textwidth}
 \hspace{-1cm}
 \includegraphics[scale=0.4]{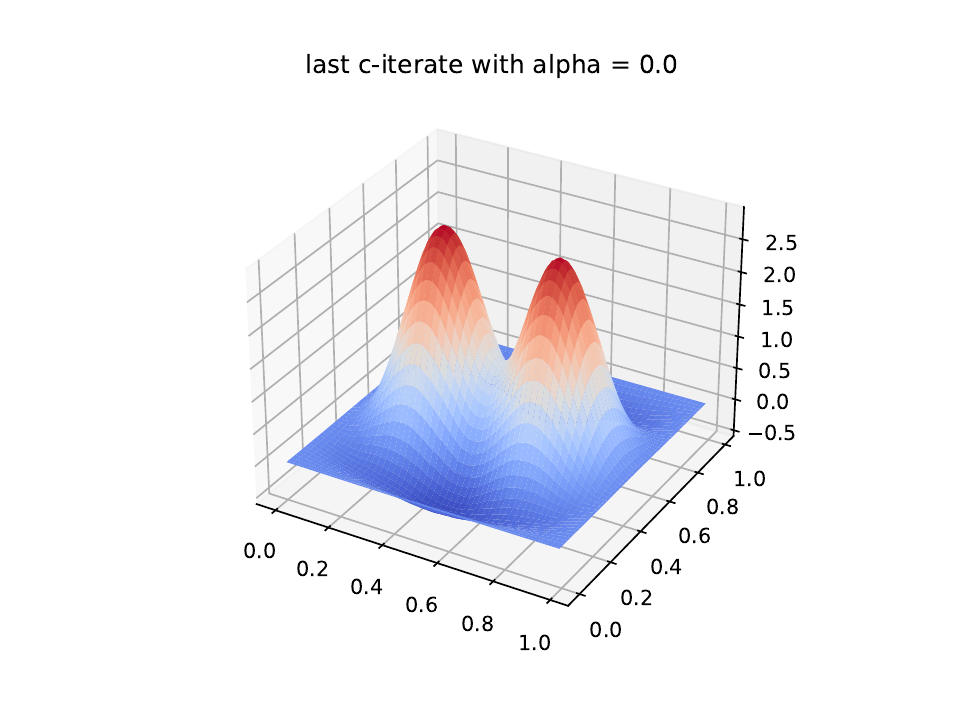}
 \end{minipage}
 \begin{minipage}{.32\textwidth}
 \hspace{-.75cm}
 \includegraphics[scale=0.4]{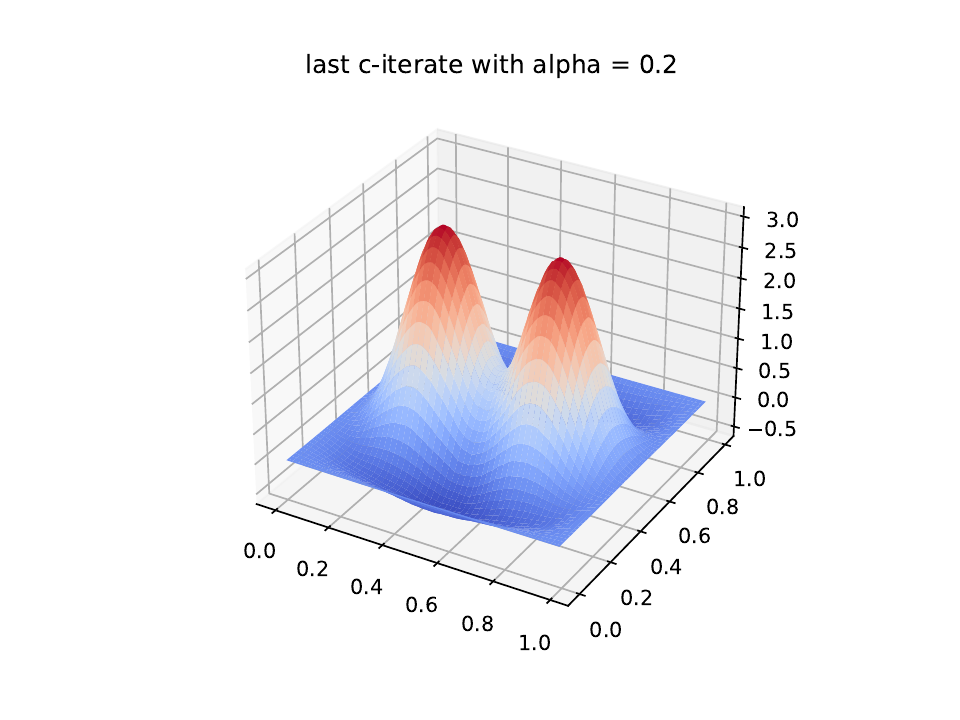}
 \end{minipage}
 \begin{minipage}{.32\textwidth}
 \hspace{-.5cm}
 \includegraphics[scale=0.4]{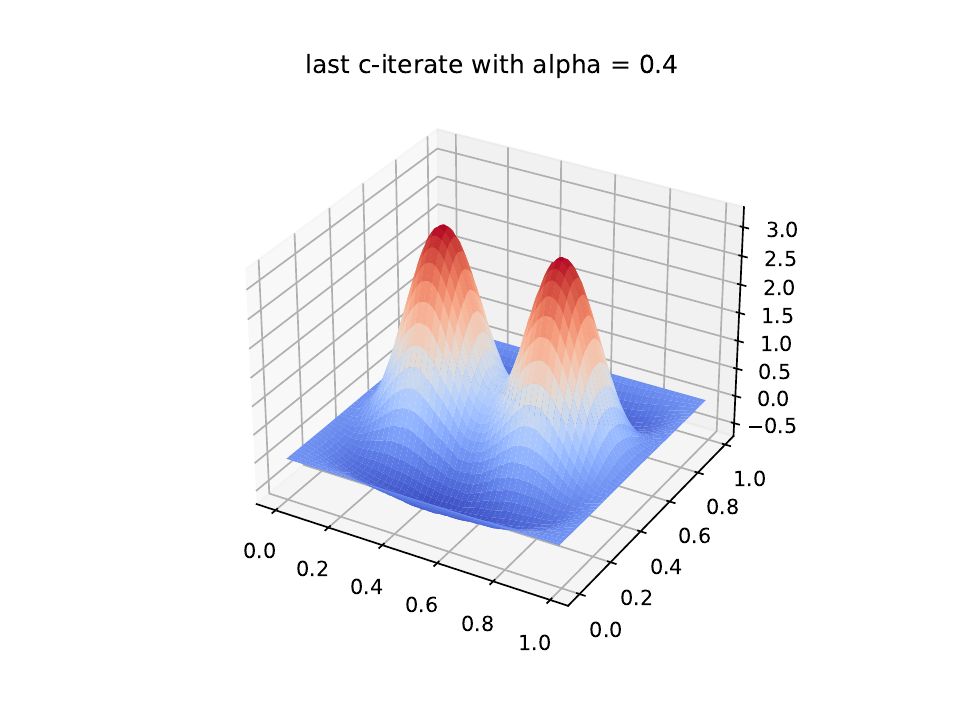}
 \end{minipage}
 \\
 \begin{minipage}{.32\textwidth}
 \hspace{-1cm}
 \includegraphics[scale=0.4]{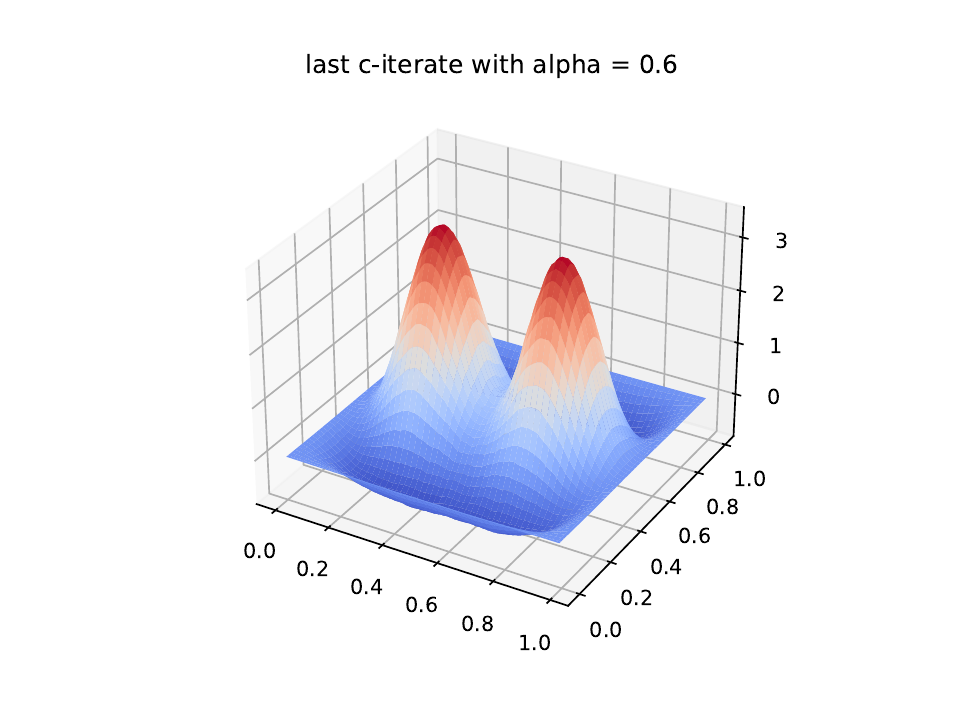}
 \end{minipage}
 \begin{minipage}{.32\textwidth}
 \hspace{-.75cm}
 \includegraphics[scale=0.4]{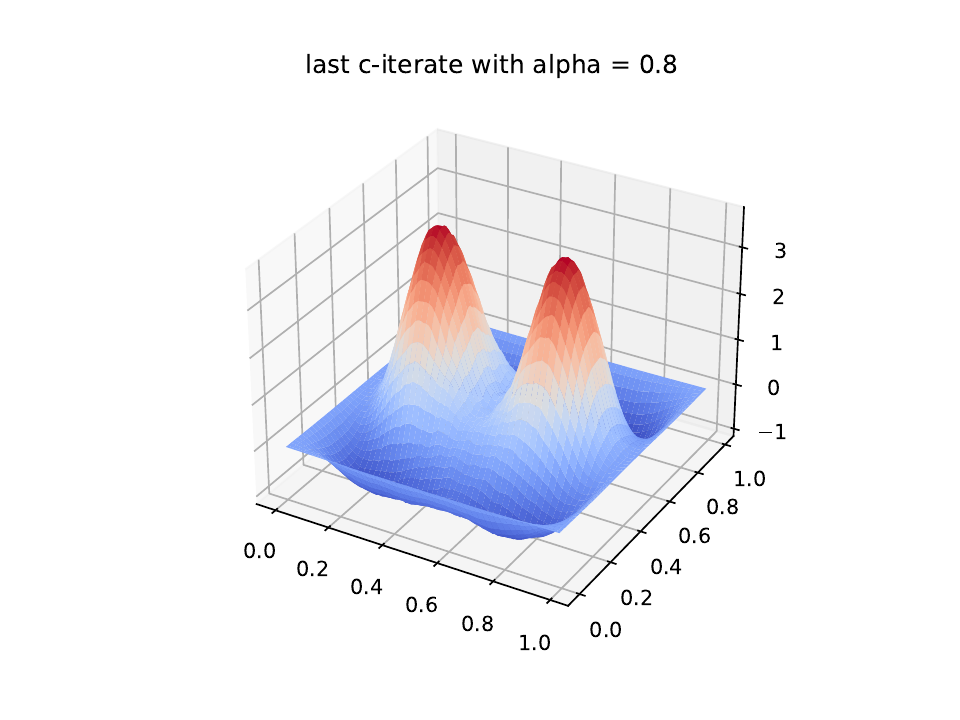}
 \end{minipage}
 \begin{minipage}{.32\textwidth}
 \hspace{-.5cm}
 \includegraphics[scale=0.4]{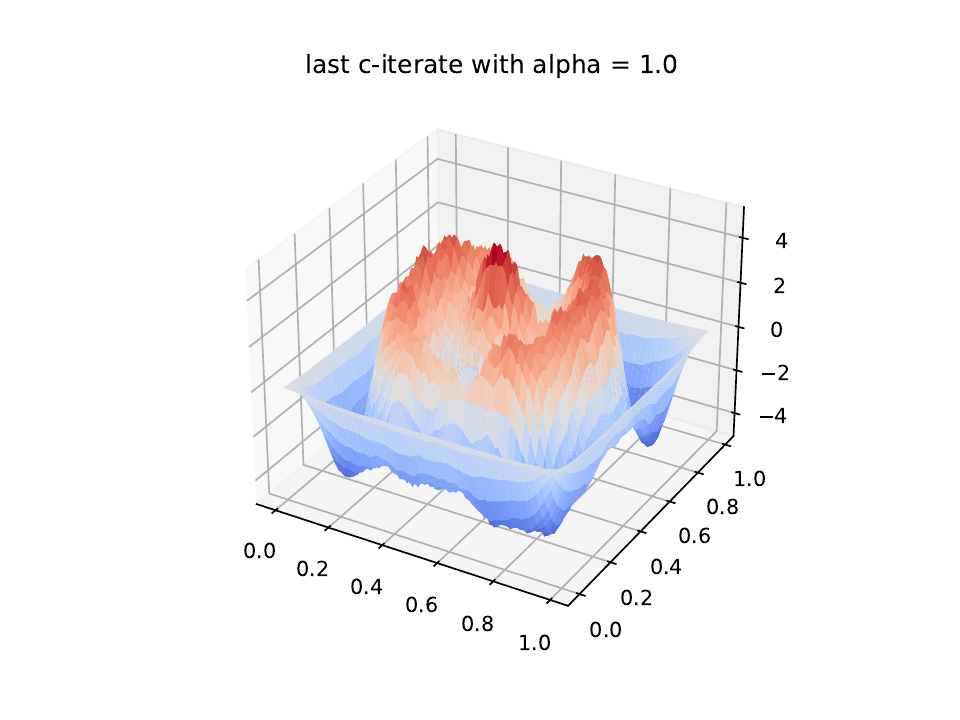}
 \end{minipage}
 \caption{Results with $1\%$ of relative noise, left to right: iterates
 $w_{100}$ for $\alpha = k/10$ for $k=0,2,...,10$}
 \label{fig:noisy_recos_100_iter}
 \end{figure}
 
 \begin{figure}[H]
 \begin{minipage}{.32\textwidth}
 \hspace{-1cm}
 \includegraphics[scale=0.4]{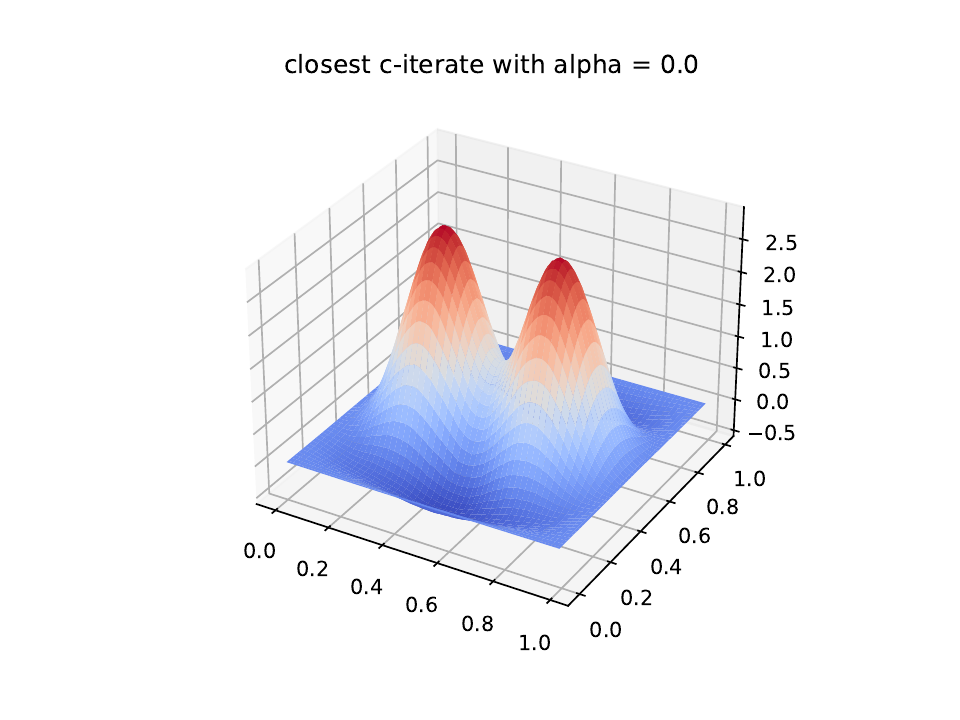}
 \end{minipage}
 \begin{minipage}{.32\textwidth}
 \hspace{-.75cm}
 \includegraphics[scale=0.4]{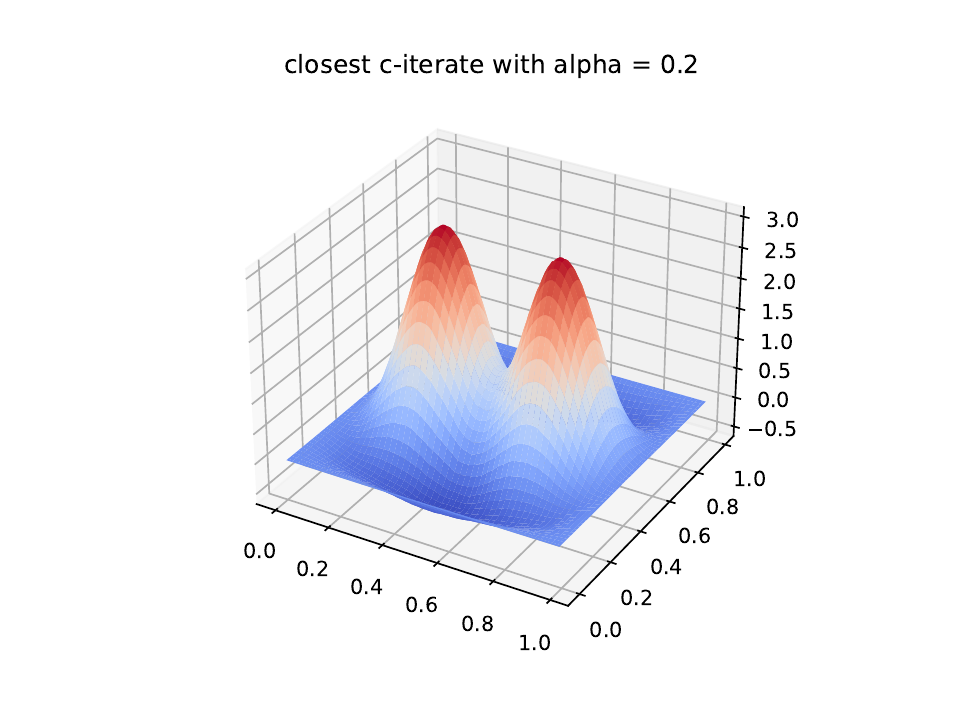}
 \end{minipage}
 \begin{minipage}{.32\textwidth}
 \hspace{-.5cm}
 \includegraphics[scale=0.4]{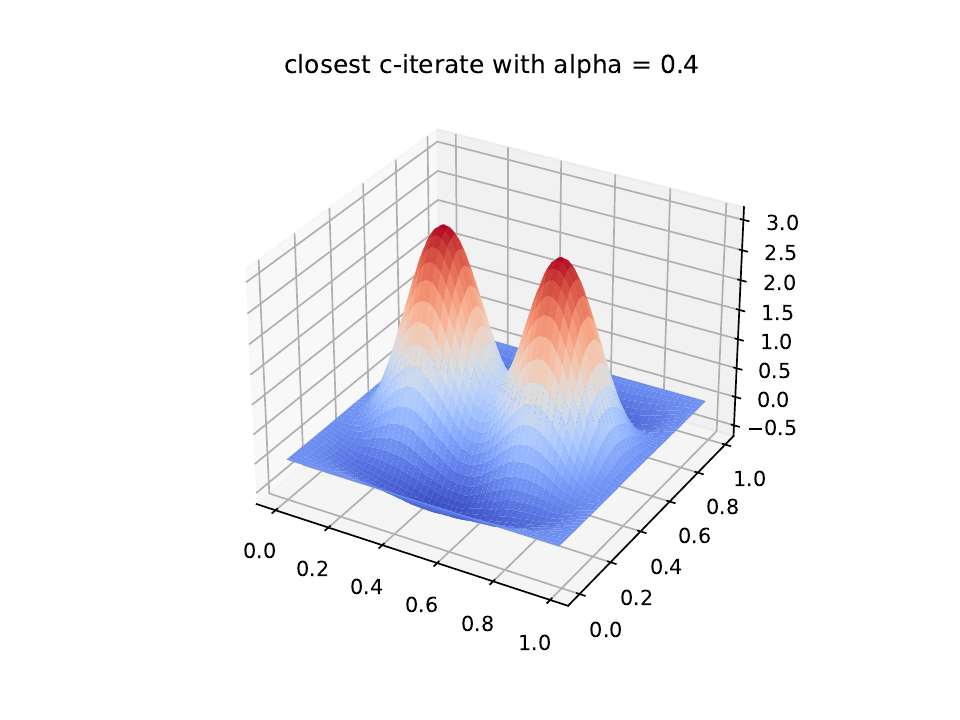}
 \end{minipage}
 \\
 \begin{minipage}{.32\textwidth}
 \hspace{-1cm}
 \includegraphics[scale=0.4]{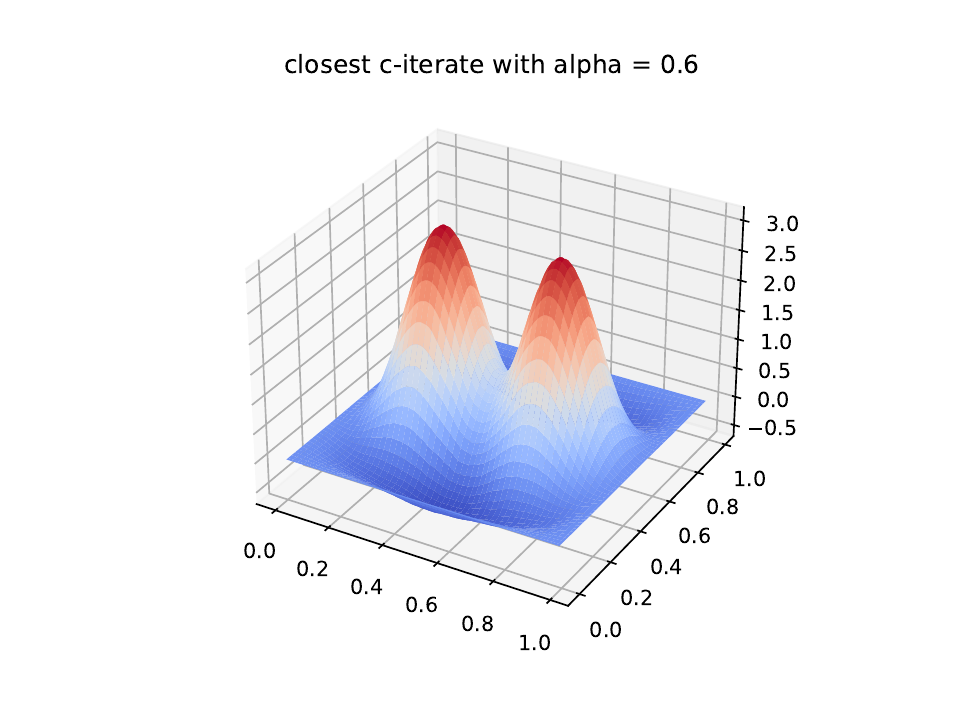}
 \end{minipage}
 \begin{minipage}{.32\textwidth}
 \hspace{-.75cm}
 \includegraphics[scale=0.4]{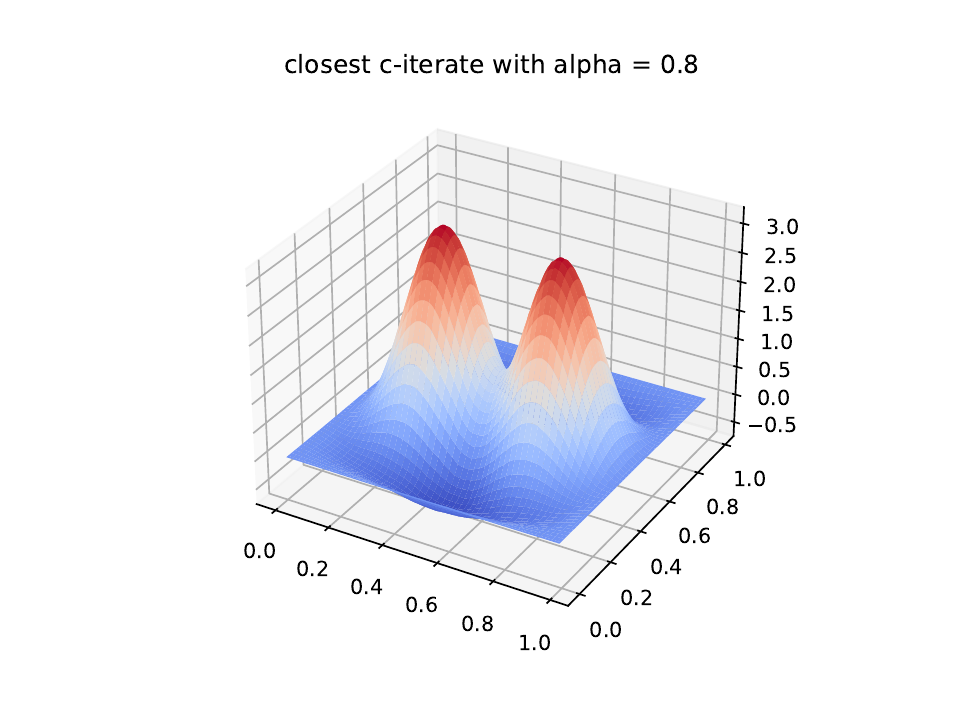}
 \end{minipage}
 \begin{minipage}{.32\textwidth}
 \hspace{-.5cm}
 \includegraphics[scale=0.4]{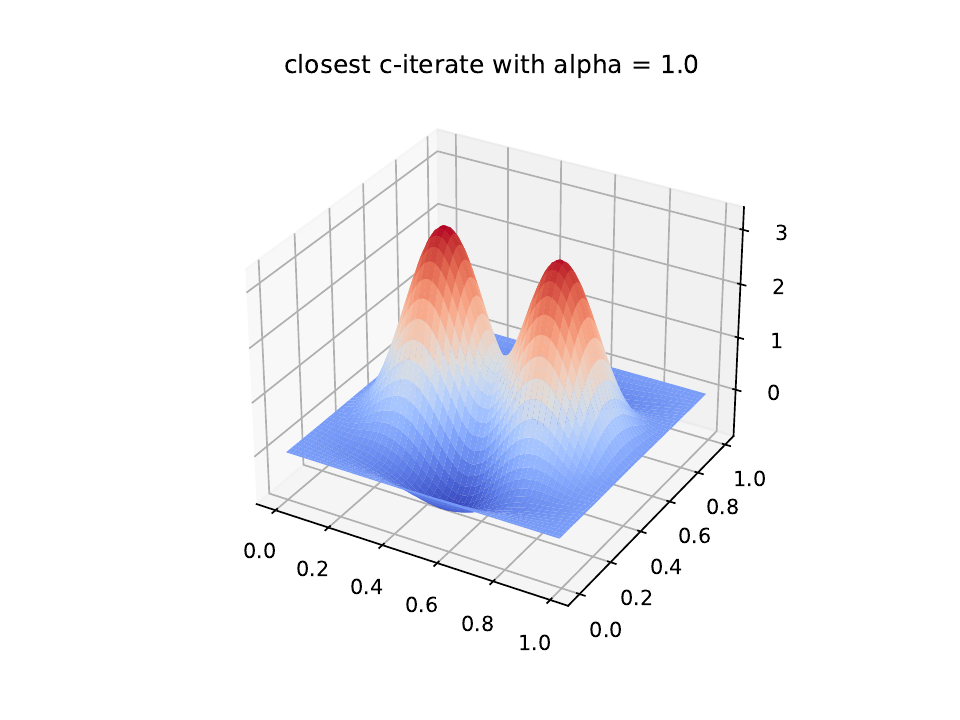}
 \end{minipage}
 \caption{Results with $1\%$ of relative noise, left to right: iterates $w_{k^*(\alpha)}$
 for $\alpha = k/10$ for $k=0,2,...,10$, where $k^*(\alpha)$ is the index where $w_k$
 is closest to $c^\dagger$ in Euclidean norm}
 \label{fig:noisy_recos_100_iter_closest}
 \end{figure}

\subsection{An inverse problem in neural network training}
\label{ssec:num-nnp}

In this section, the problem of forecasting the concentration of CO in a
gas sensor array is considered. Since we already used this model problem
for numerical experiments in \cite{RSL22, RSL24}, we are here brief in
the description.

We utilize a dataset obtained from a gas delivery platform facility at
the ChemoSignals Laboratory in the BioCircuits Institute at the University
of California, San Diego (the actual data utilized here can be accessed
on the UC Irvine Machine Learning Repository at
\href{https://archive.ics.uci.edu/ml/index.php}
{https://archive.ics.uci.edu/ml/index.php}, specifically under the dataset
titled {\em Gas sensor array under dynamic gas mixtures}).

\paragraph{Formulation of the inverse problem.}

This dataset comprises readings from 16 different chemical sensors exposed
to varying concentrations of a mixture of Ethylene and CO in the air. The
measurements were obtained through continuous acquisition of signals from
the 16-sensor array over approximately 12 hours without interruption; each
sensor data consists of $N = 4,188,262$ scalar measurements (for a
comprehensive description of the experiment, please see~\cite{FSHM15,
RSL22}).

We address the inverse problem proposed in \cite{RSL22,RSL24} namely, to
predict the reading from sensor \#16, the last sensor, by leveraging the
readings from the preceding sensors (see \cite[Figure~3]{RSL22} for
scatter plots of sensor $\#i$ readings against sensor \#16 readings, for
$i \leq 15$).
As in \cite{RSL24}, we employ a neural network (NN) in this context, which
takes the readings from the first sensors as input and produces a scalar
value as output, predicting the reading of the last sensor. Following
\cite{RSL24}, the structure of the NN used in our experiments reads:

--- Input: $z \in \R^{14}$, readings of the first 14 sensors;%
    \footnote{Sensor \#2 readings are excluded due to significant lack
    of accuracy; see \cite{RSL22}.}

--- Output: $NN(z; W, b) = \sigma(Wz + b) \in \R$.

\noindent
Here $W \in R^{1\times14}$ is a matrix of weights, $b \in R$ is a scalar
bias, and $\sigma : \R \to \R$ is the activation function defined by
\begin{equation} \label{eq:sigma}
\sigma(t) = \left\{ \begin{array}{cl}
             c+a(t-c) & , \ t \geq c \\
                    t & , \ -c < t < c \\
            -c+a(t+c) & , \ t \leq -c
            \end{array} \right. ,
\end{equation}
where $0 < a < 1$ and $c > 0$. This is a variation of the {\em saturated
linear activation function} \cite{CoPe20} (the constants $a$ and $c$
should be chosen s.t. the range of $\sigma$ contains all possible readings
of sensor \#16).

This is a shallow NN with only one layer (the output layer); the
dimention of the corresponding parameter space is 15, the dimention
of $(W,b)$. For linear $\sigma$, this NN approach simplifies to the
multiple linear regression approach considered in \cite{RSL22}.
The inverse problem under consideration is a NN training problem,
i.e. one aims to find an approximate solution to the nonlinear system
\\[1ex]
\centerline{$F_i(W,b) = y_i^\delta \, , \ \ i = 0, \dots , N_t - 1$,}
\\[1ex]
where $F_i (W, b) := NN (z_i ; W, b) = \sigma(W z_i + b)$. Here $N_t < N$
is the size of the training set and $z_i \in \R^{14}$ contains the readings
of sensors $(\#1, \#3, \#4, \dots, \#15)$, for $i = 0, \dots , N_t - 1$.
To suit our objectives, it is advantageous to express the preceding system
in the form
\begin{equation} \label{eq:ip-nn}
\mathbf{F}(W,b) \ = \mathbf{y}^\delta \, ,
\end{equation}
where $\mathbf{F}(W,b) := \big[ F_i(W,b) \big]_{i=0}^{N_t-1}$ and
$\mathbf{y}^\delta = \big[ y_i^\delta \big]_{i=0}^{N_t-1}$.

\begin{remark}[On the choice of the activation function] \label{rem:wTCC}
The real function $\sigma(t)$ in \eqref{eq:sigma} is not differentiable
at $t=-c$ and $t=c$. Consequently, the theoretical findings discussed in
Section~\ref{sec:iteration} cannot be applied to the inverse problem in
\eqref{eq:ip-nn} (indeed, the operator $\mathbf{F}$ does not satisfy (A1),
(A2)). However, one observes that:

Defining $s(t) := \partial_+\sigma(t)$, the right derivative of $\sigma$
at $t \in \R$, a direct calculation shows that
\begin{equation} \label{eq:tcc-s}
\norm{\sigma(t') - \sigma(t) - s(t)(t' - t)} \ \leq \ \widetilde{\eta}
\norm{\sigma(t') - \sigma(t)} \ \ {\it for\ all} \ t, \, t' \in \R
\end{equation}
with $\widetilde{\eta} = (1-a) a^{-1}$. 
Therefore, for each $0 \leq i < N_t$ the operator $F_i: (W,b)
\mapsto \sigma(W z_i+b)$, with $\sigma$ as in \eqref{eq:sigma},
satisfies (A2) in $\R^{14} \times \R$ with $F_i'(W,b)$ replaced by
$\widetilde{F}'_i(W,b): \R^{14} \times \R \ni (W_h,b_h) \mapsto
s(W z_i + b) (W_h z_i + b_h) \in \R$;
the corresponding constant in (A2) reads $\eta_i = (1-a) a^{-1}
\max\{ \norm{z_i} , 1 \}$.
An immediate consequence of these facts is that $\mathbf{F}:
(W,b) \mapsto \big[ F_i(W,b) \big]_{i=0}^{N_t-1}$ in \eqref{eq:ip-nn}
satisfies (A2) in $\R^{14} \times \R$, with $F_i'$ replaced by
$\widetilde{F}'$, for $\eta = \max_i \{ \eta_i \}$.

It is well known that convergence proofs of nonlinear Landweber and
LM methods can be derived under assumption (A2) where $F'$ does not
necessarily have to be the derivative of $F$ (see \cite{KalNeuSch08});
it only needs to be a linear operator that is uniformly bounded in
a neighborhood of the initial guess $x_0$. 
We conjecture that the results obtained in Section~\ref{sec:iteration}
can be extended to the framework described above. This is part of
our ongoing work.
\end{remark}

For a given pair of parameters $(W,b)$, the performance $\mathcal{P}$
of the corresponding neural network $N\!N(\cdot;W,b)$ is defined by
\begin{equation} \label{def:perf}
\mathcal{P}(N\!N(\cdot;W,b)) \ := \ 1 - \frac1{N_T} \summ_{i=N_t}^{N_t+N_T-1}
\D\frac{\norm{ N\!N(z_i; W,b) - y_i^\delta}}{\norm{y_i^\delta}} \, ,
\end{equation}
were $N_T \in \N$ is the size of the test set.
The sum in the above definition gives the average misfit betwen the
predicted value $N\!N(z_i;W,b)$ and $y_i^\delta$, evaluated over the
test set $\{ z_i , \ N_t \leq i < N_t + N_T - 1\}$.
Notice that $0 \leq \mathcal{P}(N\!N(\cdot;W,b)) \leq 1$ for all $(W,b)$,
while $\mathcal{P}(N\!N(\cdot;W,b)) = 1$ is the best possible performance.

\begin{remark}[The training set and test set]
The 'training set' and 'test set' are comprised of samples with sizes
of $N_t$ and $N_T$ respectively.
In our numerical experiments we use $N_t = 4,000,000$ and $N_T = 100,000$
(notice that $Nt + N_T < N$).
\end{remark}

\paragraph{Numerical implementations.}

In what follows the inLM method is implemented for solving the NN
training problem \eqref{eq:ip-nn}.
In view of Remark~\ref{rem:wTCC} we choose $a = \frac23$ and $c=8$ in
\eqref{eq:sigma}. Consequently, we generate an activation function
$\bar\sigma$ of the form \eqref{eq:sigma}, satisfying \eqref{eq:tcc-s}
for $\eta = 0.5$.

The sensor readings $(z_i, y_i^\delta) \in \R^{14} \times \R$ on the
training set are scaled by the factor $\max_{i\leq N_t} \norm{z_i}$.
An analogous procedure is performed on the test set.
Consequently, after scaling, it holds $\norm{z_i} \leq 1$, for
$i = 0, \dots, N_t+N_T$.
From Remark~\ref{rem:wTCC} it follows that, for $\bar\sigma$ as above,
all operators $F_i(W,b)$ satisfy (A2) (with $F'_i$ replaced by
$\widetilde{F}'_i$) for the same constant $\eta = 0.5$; the same
holds for the operator $\mathbf{F}$ in \eqref{eq:ip-nn}.

In our experiments the initial guess $(W_0,b_0)$ is a random vector with
coordinate values ranging in $(-1,1)$.
We approximate the linear solve in step [2.1] of Algorithm~\ref{alg:init-noise}
by three steps of the conjugate gradient method with zero initial value.
Three different runs of the inLM method are presented, each one for
a different choice of (constant) inertial parameter $\alpha_k$,
namely $\{ 0.05, 0,10, 0.20 \}$.

\begin{figure}[h]
\vspace{-1.2cm}
\centerline{\includegraphics[width=0.8\textwidth]{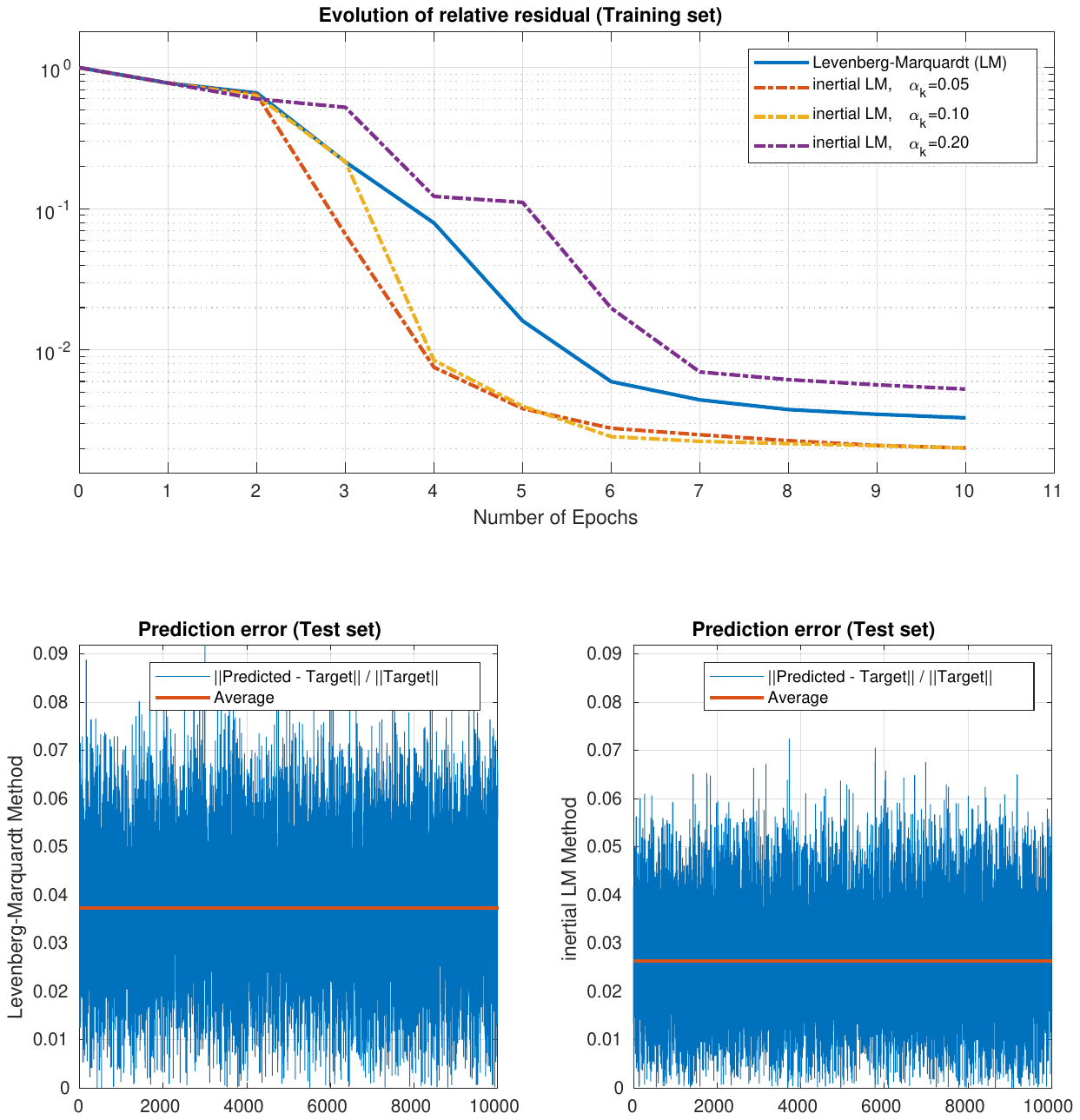}}
\vskip-2.2cm
\caption{\small Neural Network training problem.
(TOP) Evolution of relative residual for different methods;
(BOTTOM) Prediction error of the trained NN for the test-set:
LM method (left) and inLM method (right).}
\label{fig:inertial-LM}
\end{figure}

For comparison purposes the classical LM method ($\alpha_k = 0$) was
also implemented.
Since the noise level $\delta$ is not known, all methods are computed
for ten steps;\footnote{Each step corresponds to an epoch.}
after the tenth step the residual evolution stagnates for all methods.
The obtained results are summarized in Figure~\ref{fig:inertial-LM}:

\begin{itemize}
\item (TOP) Evolution of relative residual
  \,$\sum_{i=0}^{N_t-1} \frac{\norm{N\!N(z_i; W_k,b_k) - y_i^\delta}}
  {\norm{N\!N(z_i; W_0,b_0) - y_i^\delta}}$ on the training set -- all
  methods.

\item (BOTTOM-RIGHT) inLM method: relative prediction error
  $\norm{N\!N(z_i; W_{k^*},b_{k^*}) - y_i^\delta} / \norm{y_i^\delta}$ is
  plotted for the test set $\{ z_i , \ N_t \leq i < N_t + N_T - 1\}$ (BLUE);
  the average value (RED) is $0.026$. The performance of the trained
  Neural Network amounts to $\mathbf{97}$\%. \medskip
\item  (BOTTOM-LEFT): For comparison, the prediction accuracy
  of the NN trained by the LM method is plotted for the same test set (BLUE),
  the average value is $0.037$ (RED). The performance of the trained
  Neural Network amounts to $\mathbf{96}$\%. 
\end{itemize}
Here are a few observations from our numerical experiments:

\begin{itemize}
\item For constant choices of $\alpha_k$, small values yield the best
  results (in our experiments $\alpha_k=0.05$ and $0.10$).
  For even smaller constant values, such as $\alpha_k=0.01$, the performance
  of the inLM method becomes very similar to that of the LM method
  (which corresponds to $\alpha_k=0$).

\item For larger constant values of $\alpha_k$, e.g. $\alpha_k=0.20$,
  the inLM method becomes unstable and its performance deteriorates
  compared to that of the LM method.

\item The Neural Network trained using the inLM method outperforms
  the one trained with the LM method.
  Additionally, the inLM method converges faster. The residual decay for
  the inLM method stagnates after 6 steps, whereas it takes 10 steps
  for the LM method (see Figure~\ref{fig:inertial-LM}).
\end{itemize}

\section{Final remarks and conclusions} \label{sec:conclusions}

In this manuscript we propose and analyze an implicit inertial type
iteration, namely the {\em inertial Levenberg Marquardt} (inLM) method,
as an alternative for obtaining stable approximate solutions to
nonlinear ill-posed operator equations.
This new method can be considered as an extension of the classical
Levenberg Marquardt (LM) method (indeed, if the inertial parameters
$\alpha_k$ are set to zero the inLM reduces to the LM method).

The main results discussed in this notes are: boundedness of the
sequences $(x_k)$ and $(w_k)$ generated by the inLM method
(Propositions~\ref{prop:gain} and~\ref{prop:bound}), strong convergence
for exact data (Theorem~\ref{th:conv}), stability and semi-convergence
for noisy data (Theorems~\ref{th:stabil} and~\ref{th:semi-conv} respectively).
We also provide a bound for the stopping index in the noisy data case
(Proposition~\ref{prop:kstar}).

In Section~\ref{sec:numerics} two distinct ill-posed problems are used
to investigate the numerical efficiency of the proposed inLM method:
A parameter identification problem in an elliptic PDE and an inverse
problem in neural network training.

The preliminary results obtained in our numerical experiments indicate
a better performance (faster convergence) of the inLM method when
compared to the LM method.
The inLM method not only converges faster than the LM method (as shown
in Figures~\ref{fig:noisy_res_and_distance} and~\ref{fig:inertial-LM}),
but it also attains an approximate solution with a significantly smaller
residual in the second inverse problem.

\section*{Acknowledgments}

AL acknowledges support from the AvH Foundation. Significant part of
this manuscript was writen while this author was on sabbatical leave
at EMAp, Getulio Vargas Fundation, Rio de Janeiro, Brazil.
DAL acknowledges support from the AvH foundation.



\bibliographystyle{amsplain}
\bibliography{inertLM}

\end{document}